\documentclass[12pt,reqno]{amsproc}
\usepackage{mathrsfs, amsfonts, amsthm, amsmath, amssymb}

 \theoremstyle{plain}

\theoremstyle{plain}
\newtheorem{theorem}{Theorem}[section]

\newtheorem{corollary}[theorem]{Corollary}

\newtheorem{definition}[theorem]{Definition}
\newtheorem{example}[theorem]{Example}

\newtheorem{lemma}[theorem]{Lemma}
\newtheorem{notation}[theorem]{Notation}
\newtheorem{problem}[theorem]{Problem}
\newtheorem{proposition}[theorem]{Proposition}
\newtheorem{remark}[theorem]{Remark}

\makeatletter
\newcommand{\LeftEqNo}{\let\veqno\@@leqno}
\makeatother

 \numberwithin{equation}  {section}
\begin{document}
\title[A Beurling-Blecher-Labuschagne  theorem for semifinite Hardy spaces]{A Beurling-Blecher-Labuschagne   theorem \\ for noncommutative  Hardy spaces associated \\ with semifinite von Neumann algebras}

\author{Lauren Sager}
\address{Lauren Sager \\
        Demartment of Mathematics and Statistics \\
         University of New Hampshire\\
         Durham, NH 03824;   Email: lbq32@wildcats.unh.edu}

\begin{abstract}
 In 2008, Blecher and Labuschagne extended Beurling's classical theorem
  to $H^\infty$-invariant subspaces of  $L^p(\mathcal{M},\tau)$ for a finite von Neumann algebra $\mathcal{M}$ with a finite,
  faithful, normal tracial state $\tau$ when $1\le p\le \infty$. 
   In this paper,   using Arveson's non-commutative Hardy space  $H^\infty$ in relation to a von Neumann algebra
   $\mathcal{M}$
   with a semifinite, faithful, normal tracial weight $\tau$, we prove a Beurling-Blecher-Labuschagne theorem
   for $H^\infty$-invariant spaces of $L^p(\mathcal{M},\tau)$ when $0<p\leq\infty$.  The proof of the main result
   relies on proofs of density theorems for $L^p(\mathcal{M},\tau)$ and semifinite versions of several other known
   theorems from the finite case.  
     Using the main result, we are   able to completely characterize  all $H^\infty$-invariant
      subspaces of $L^p(\mathcal M\rtimes_\alpha \mathbb Z,\tau)$, where $\mathcal M\rtimes_\alpha \mathbb Z $
      is a
      crossed product  of a semifinite von Neumann algebra $\mathcal{M}$ by the integer group $\mathbb Z$ and $H^\infty$ is a non-selfadjoint crossed product  of   $\mathcal{M}$ by   $\mathbb Z^+$ . As an example, we   characterize
       all $H^\infty$-invariant
      subspaces of the Schatten $p$-class $S^p(\mathcal{H})$, where $H^\infty$ is the lower triangular subalgebra of $B(\mathcal H)$, for
      each $0<p\leq\infty$.
\end{abstract}


\maketitle

\section{Introduction}

Let $\mathcal H$ be an infinite dimensional  separable Hilbert space with an orthonormal base  $\{e_m\}_{m\in\mathbb Z}$, and   $B(\mathcal H)$ be the set of all bounded linear operators on $\mathcal {H}$.
 Let $\tau=Tr$ be the usual trace on $B(\mathcal H)$, i.e.
  $$
  \tau(x)=\sum_{i\in\mathbb Z} \langle xe_m, e_m\rangle , \qquad \text{ for all positive } x   \text { in }    B(\mathcal H).
  $$
 For each $0<p<\infty$,  the Schatten $p$-class $S^p(\mathcal H)$ consists all these elements $x$ in $B(\mathcal H)$ such that
 $\tau(|x|^p)<\infty$. It is well-known that $S^p(\mathcal H)$ is a complete metric space (a Banach space when $p\ge 1$ and a Hilbert space when $p=2$). Moreover, $S^p(\mathcal H)$ is  a two sided ideal of $B(\mathcal H)$.

Let
$$
\mathcal A =\{ x\in B(\mathcal H) : \langle xe_m, e_n\rangle =0, \ \ \forall n<m\}
$$ be the lower triangular subalgebra of $B(\mathcal H)$. In the paper, we are interested in answering the following
question, which is implicitly asked by   McAsey,  Muhly and   Saito
in Example 2.6 of \cite{MMS}.

\begin{problem} \label{prob1.1}
Given a closed subspace $\mathcal{K}$ of the Schatten p-class $S^p(\mathcal{H})$ where $0<p<\infty$, such that  $\mathcal{K}$ satisfies $\mathcal{AK}\subseteq\mathcal{K}$, how can we characterize the subspace $\mathcal{K}$?
\end{problem}

The answer to Problem \ref{prob1.1} is closely related to our
generalization of the classical Beurling theorem for a Hardy space.
Recall  the classical Beurling theorem for invariant subspaces as
follows. Let $\mathbb{T}$ be the unit circle, and let $\mu$ be the
measure on $\mathbb{T}$ such that $d\mu=\frac{1}{2\pi}d\theta$.  Let
$L^\infty(\mathbb{T},\mu)$ be the commutative von Neumann algebra on
$\mathbb{T}$, and define
$L^2(\mathbb{T},\mu)$ to be the closure of $L^\infty(\mathbb{T},
\mu)$ under the $\|\cdot\|_2$.  Then $L^2(\mathbb{T},\tau)$ is a Hilbert space with
orthonormal basis $\{z^n:n\in\mathbb{Z}\}$.  Let
$\displaystyle H^2=\overline{span\{z^n:n\ge 0\}}^{\|\cdot \|_2}\subseteq L^2((\mathbb{T},\mu)$ and $\displaystyle H^\infty =H^2\cap L^\infty(\mathbb{T},
\mu)$.  If we define $M_\psi(f)=\psi f$
for every $f\in L^2(\mathbb{T},\mu)$, it is easy to show that
$L^\infty(\mathbb{T},\mu)$ has a representation onto
$\mathcal{B}(L^2(\mathbb{T},\mu))$ by the map $\psi\rightarrow
M_\psi$.  Therefore,
$L^\infty(\mathbb{T},\mu)$ and $H^\infty$ can be assumed to act
naturally on $L^2(\mathbb{T},\mu)$ by multiplication on the left (or
right).  The classical Beurling's theorem, proven in 1949 by A.
Beurling in \cite{B}, states the following: {\em  If $\mathcal{W}$
is a nonzero closed, $H^\infty$-invariant subspace (or,
equivalently, $z\mathcal{W}\subseteq \mathcal{W}$ for every $z\in H^\infty$) of $H^2$, then
$\mathcal{W}=\psi H^2$ for some $\psi\in H^\infty$ with $|\psi |=1$
$a.e.(\mu)$. }

Then we define $L^p(\mathbb{T},\tau)$ to be the closure of $L^\infty(\mathbb{T},\tau)$ under $\|\cdot\|_p$, and $H^p=\{f\in L^p(\mathbb{T},\mu) : \int_\mathbb{T} f(e^{i\theta})e^{in\theta} d\mu(\theta) = 0$ for $n\in\mathbb{N}\}$ for $1\leq p<\infty$.  Beurling's theorem has been extended for $H^\infty$-invariant subspaces in Hardy spaces $H^p$ for $1\leq p \leq \infty$.  (See, for example, \cite{Bo}, \cite{Halmos}, \cite{He}, \cite{HL}, \cite{Ho}, \cite{Sr}, and others).  The classical Beurling's theorem has been extended in many other ways as well.

One such extension comes from the work of  D. Blecher and L.
Labuschagne in \cite{BL2}.  We recall the construction of
$L^p(\mathcal{M},\tau)$.  Let $\mathcal{M}$ be a semifinite von
Neumann algebra, and let $\tau$ be a faithful, normal tracial weight
on $\mathcal{M}$ (when $\tau(I)<\infty$, $\mathcal{M}$ is finite).
Let $\mathcal I$ be the set of elementary operators in $\mathcal M$ (when $\mathcal{M}$ is finite, $\mathcal I=\mathcal M$).
Then define a mapping from $\mathcal{I}$ to $[0,\infty)$ by
$\|x\|_p=(\tau(|x|^p))^{1/p}$ for every $x\in \mathcal I$, and where
$|x|=\sqrt{x^*x}$.  It is nontrivial to prove that when $1\leq
p<\infty$, $\|\cdot\|_p$ defines a norm on $\mathcal{I}$, which we
call the p-norm.  We may then define
$L^p(\mathcal{M},\tau)=\overline{\mathcal{I}}^{\|\cdot\|_p}$.  We
let $L^\infty(\mathcal{M},\tau)=\mathcal{M}$, and this space acts
naturally on $L^p(\mathcal{M},\tau)$ by left (or right)
multiplication.

We then recall the definition of Arveson's non-commutative Hardy space from \cite{Arv}.  If $\mathcal{M}$ is a von Neumann algebra, with faithful, normal tracial weight $\tau$, let $\mathcal{A}\subseteq \mathcal{M}$ be a weak* closed subalgebra.  Then let $\mathcal{D}=\mathcal{A}\cap \mathcal{A}^*$ be a von Neumann subalgebra of $\mathcal{M}$.  Then there exists $\Phi:\mathcal{M}\rightarrow \mathcal{D}$, a faithful, normal, trace-preserving conditional expectation, which can be extended to $\Phi:L^1(\mathcal{M},\tau)\rightarrow L^1(\mathcal{D},\tau)$.  Then $\mathcal{A}$ is called a non-commutative Hardy space if (1) $\Phi(xy)=\Phi(x)\Phi(y)$ for every $x,y\in\mathcal{A}$; (2) $\mathcal{A}+\mathcal{A}^*$ is weak* dense in $\mathcal{M}$; (3) $\tau(\Phi(x))=\tau(x)$ for every $x\in \mathcal{M}$.

Blecher and Labuschagne proved the following theorem for finite von
Neumann algebras in \cite{BL2}. {\em  Let $\mathcal M$ be a finite
von Neumann algebra with a faithful, tracial, normal state $\tau$,
and $ H^{\infty}$ be a maximal subdiagonal subalgebra of
$\mathcal{M}$ with $\mathcal D=H^\infty\cap (H^\infty)^*$. Suppose
that $\mathcal K$ is a closed $H^{\infty}$-right-invariant subspace
of $L^p(\mathcal{M},\tau),$ for some $1\leq p\leq \infty.$ (For
$p=\infty$ it is assumed that $\mathcal K$ is weak* closed.) Then
$\mathcal K$
may be written as a column $L^{p}$-sum $\mathcal K= \mathcal Z  \oplus^{col}(\oplus^{col}%
_{i}u_{i}H^{p})$, where $ \mathcal Z $ is a closed (indeed weak*
closed if $p=\infty$)
  subspace of $L^{p}(\mathcal{M},\tau)$ such that $ \mathcal Z =[ \mathcal Z H_{0} ^{\infty}]_{p}$, and where $u_{i}$ are partial
isometries in ${\mathcal{M}}\cap \mathcal K$ satisfying certain
conditions (For more details, see \cite{BL2}.})Here $\oplus^{col}%
_{i}u_{i}H^{p}$ and $ \mathcal Z =[ \mathcal Z H_{0} ^{\infty}]_{p}$
are of type 1, and  type 2 respectively (also see \cite{NW} for
definitions of invariant subspaces of different types).


Examples of finite von Neumann algebras include the spaces
$M_n(\mathbb{C})$ of all $n\times n$ matrices with  complex entries
when $1\leq n<\infty$.  However, if   $\mathcal H$ is an infinite
dimensional separable Hilbert space  and we   view $B(\mathcal H)$ as
$M_\infty(\mathbb{C})$,  the set of all (bounded) $\infty \times
\infty$ matrices with complex entries,  then  $B(\mathcal H)$ is a
semifinite von Neumann algebra, and   no longer satisfies the
hypothesis of the Beurling-Blecher-Labuschagne theorem.

In this paper, we therefore consider a version of Blecher and Labuschagne's Beurling's theorem for semifinite von Neumann algebras.  We seek to characterize $H^\infty$-invariant spaces of  $L^p(\mathcal{M},\tau)$ spaces.  Adapting Blecher and Labuschagne's theorem to the semifinite case, we prove the following result:



{\renewcommand{\thetheorem}{\ref{theorem4.6}}
    \begin{theorem}  Let $1\le p<\infty$.
Let $\mathcal M$ be a von Neumann algebra with a faithful, normal,
semifinite tracial weight $\tau$, and $H^\infty$  be  a semifinite
subdigonal subalgebra   of $\mathcal M$ (see Definition
\ref{Definition2.7}).  Let $\mathcal D=H^\infty\cap (H^\infty)^*$.
Assume that $\mathcal K  $ is a closed subspace of $L^p(\mathcal
M,\tau)$ such that $H^\infty \mathcal K \subseteq \mathcal K$.

             Then there exist   a   closed subspace $Y$ of $L^p(\mathcal M,\tau)$ and a family $\{u_\lambda\}_{\lambda \in\Lambda}$ of partial isometries in
             $ \mathcal{M}$  such that:
            \begin{enumerate}
              \item  [(i)] $u_\lambda Y^*=0$ for all ${\lambda\in\Lambda}$.
\item  [(ii)] $u_\lambda u_\lambda^*\in \mathcal D$ and $u_\lambda u_\mu^*=0$
for all $\lambda, \mu\in \Lambda$ with $\lambda\ne \mu$.
\item  [(iii)] $Y= [H^\infty_0Y]_p$.
                \item [(iv)] $\mathcal K=Y \oplus^{row} (\oplus^{row}_{\lambda\in\Lambda } H^p u_\lambda)$
            \end{enumerate} Here $\oplus^{row}$ is the row sum of
            subspaces
            defined in Definition \ref{definition2.15}.
        \end{theorem}
}

However, many of the methods used by Blecher and Labuschagne do not apply directly when $\mathcal{M}$ is a semifinite von Neumann algebra.  Thus, we prove a density theorem for semifinite von Neumann algebras through a series of lemmas and propositions.

{\renewcommand{\thetheorem}{\ref{proposition4.1}}

    \begin{proposition}
Let $\mathcal M$ be a von Neumann algebra with a faithful, normal,
semifinite tracial weight $\tau$, and $H^\infty$  be  a semifinite
subdigonal subalgebra   of $\mathcal M$. Let $1\le p <\infty$.
Assume that $\mathcal K $ is  a closed subspace in $
L^p(\mathcal{M},\tau)$ such that $H^\infty \mathcal K \subseteq
\mathcal K$. Then the following statements are true.
            \begin{enumerate}
                \item [(i)]$\mathcal K\cap\mathcal{M} =\overline{\mathcal K\cap\mathcal{M}}^{w^*}\cap L^p(\mathcal{M},\tau)$.
                \item [(ii)]  $\mathcal K = [\mathcal K \cap \mathcal{M}]_p$.

            \end{enumerate}
     \end{proposition}

}

 {\renewcommand{\thetheorem}{\ref{proposition4.2}}
 \begin{proposition}
Let $\mathcal M$ be a von Neumann algebra with a faithful, normal,
semifinite tracial weight $\tau$, and $H^\infty$  be  a semifinite
subdigonal subalgebra   of $\mathcal M$.

Assume that  $\mathcal K \subseteq  \mathcal{M} $ is weak$^*$-closed
subspace such that  $H^\infty \mathcal K \subseteq \mathcal K$. Then
$$  \mathcal K=\overline{[  \mathcal K \cap L^p(\mathcal{M},\tau)]_p \cap
\mathcal{M}}^{w^*}, \quad \forall \ 1\le p<\infty.$$
     \end{proposition}
}


{\renewcommand{\thetheorem}{\ref{prop4.4}}
 \begin{proposition}
Let $\mathcal M$ be a von Neumann algebra with a faithful, normal,
semifinite tracial weight $\tau$, and $H^\infty$  be  a semifinite
subdigonal subalgebra   of $\mathcal M$.  Assume that  $  S
\subseteq \mathcal{M} $ is a subspace such that  $H^\infty  S
\subseteq  S$.   Then
          $$[S\cap L^p(\mathcal{M},\tau)]_p = [\overline{S}^{w^*} \cap L^p(\mathcal{M},\tau)]_p, \quad \forall \ 1\le p<\infty.$$

     \end{proposition}
}

Subsequently, we are able to prove a noncommutative
Beurling-Blecher-Labuschagne theorem for the semifinite case when
$0<p<1$.

{\renewcommand{\thetheorem}{\ref{theorem5.4}}
   \begin{theorem}   Let $0<p<1$.
Let $\mathcal M$ be a von Neumann algebra with a faithful, normal,
semifinite tracial weight $\tau$, and $H^\infty$  be  a semifinite
subdigonal subalgebra   of $\mathcal M$ (see Definition
\ref{Definition2.7}).  Let $\mathcal D=H^\infty\cap (H^\infty)^*$.
Assume that $\mathcal K  $ is a closed subspace of $L^p(\mathcal
M,\tau)$ such that $H^\infty \mathcal K \subseteq \mathcal K$.

             Then there exist   a   closed subspace $Y$ of $L^p(\mathcal M,\tau)$ and a family $\{u_\lambda\}_{\lambda \in\Lambda}$ of partial isometries in
             $ \mathcal{M}$  such that:
            \begin{enumerate}
              \item  [(i)] $u_\lambda Y^*=0$ for all ${\lambda\in\Lambda}$.
\item  [(ii)] $u_\lambda u_\lambda^*\in \mathcal D$ and $u_\lambda u_\mu^*=0$
for all $\lambda, \mu\in \Lambda$ with $\lambda\ne \mu$.
\item  [(iii)] $Y= [H^\infty_0Y]_p$.
                \item [(iv)] $\mathcal K=Y \oplus^{row} (\oplus^{row}_{\lambda\in\Lambda } H^p u_\lambda)$
            \end{enumerate} Here $\oplus^{row}$ is the row sum of
            subspaces
            defined in Definition \ref{definition2.15}.
        \end{theorem}
Here, we use similar methods to our proof for $1\leq
p<\infty$, including proving a similar density theorem (see
Proposition \ref{proposition5.1}, Proposition \ref{proposition5.2}).

}






Using theorems \ref{theorem4.6} and \ref{theorem5.4} and corollary \ref{corollary5.5}, we are then able to prove a Beurling-Blecher-Labuschagne-like theorem for the crossed product of a semifinite von Neumann algebra $\mathcal{M}$ by a trace-preserving action $\alpha$ when $0<p<\infty$.    We are actually able to fully characterize the $H^\infty$-invariant subspace of the crossed product.

{\renewcommand{\thetheorem}{\ref{theorem6.2}}
\begin{theorem}
 Let $\mathcal M$ be a von Neumann algebra with a semifinite, faithful, normal, tracial weight $\tau$, and  $\alpha$ be a  trace-preserving $*$-automorphism  of $\mathcal M$. Denote by  $ \mathcal M\rtimes_\alpha \mathbb Z$ the crossed product of $\mathcal M$ by an action $\alpha$, and still denote by $\tau$ a  semifinite, faithful, normal, extended tracial weight on  $\mathcal M\rtimes_\alpha \mathbb Z$.

Let     $H^\infty $,   a weak $*$-closed nonself-adjoint subalgebra generated by
$
\{\Lambda (n)\Psi(x)  : x\in \mathcal M, \ n\ge 0\}
$  in  $\mathcal M\rtimes_\alpha \mathbb Z$, be a semifinite  subdiagonal subalgebra of  $\mathcal M\rtimes_\alpha \mathbb Z$.
 Then the following statements are true.

\begin{enumerate}
\item [(i)] Let $0<p<\infty$.  Assume that $\mathcal K  $ is a closed
subspace of $L^p(\mathcal M\rtimes_\alpha \mathbb Z,\tau)$ such that $H^\infty \mathcal K
\subseteq \mathcal K$.
             Then there exist a projection $q$ in $\mathcal M$  and    a    family $\{u_\lambda\}_{\lambda \in\Lambda}$ of partial isometries in
             $ \mathcal M\rtimes_\alpha \mathbb Z$  satisfying
            \begin{enumerate}
              \item [(a)]  $u_\lambda q=0$  for all $\lambda \in \Lambda$;
              \item  [(b)]  $u_\lambda u_\lambda^*\in \mathcal M$ and $u_\lambda u_\mu^*=0$
for all $\lambda, \mu\in \Lambda$ with $\lambda\ne \mu$;
                \item [(c)] $\mathcal K=(L^p( \mathcal M\rtimes_\alpha \mathbb Z,\tau)q) \oplus^{row}(  \oplus^{row}_{\lambda\in\Lambda } H^p u_\lambda).$
            \end{enumerate}
\item [(ii)]   Assume that $\mathcal K  $ is a weak $*$-closed
subspace of $ \mathcal M\rtimes_\alpha \mathbb Z $ such that $H^\infty \mathcal K
\subseteq \mathcal K$.
             Then there  exist a projection $q$ in $\mathcal M$  and   a    family $\{u_\lambda\}_{\lambda \in\Lambda}$ of partial isometries in
             $ \mathcal M\rtimes_\alpha \mathbb Z$  satisfying
            \begin{enumerate}
             \item [(a)]  $u_\lambda q=0$  for all $\lambda \in \Lambda$;
              \item  [(b)]  $u_\lambda u_\lambda^*\in \mathcal M$ and $u_\lambda u_\mu^*=0$
for all $\lambda, \mu\in \Lambda$ with $\lambda\ne \mu$;
                \item [(c)] $\mathcal K= ( ( \mathcal M\rtimes_\alpha \mathbb Z)q )\oplus^{row}( \oplus^{row}_{\lambda\in\Lambda } H^\infty u_\lambda).$
            \end{enumerate}
\end{enumerate}

\end{theorem}
}

In \cite{MMS}, McAsey, Muhly and Saito prove a Beurling's theorem for a crossed product. \emph{Suppose $\mathcal{M}$ is a finite von Neumann algebra with finite trace $\tau$ and $\alpha$, a trace preserving automorphism of $\mathcal{M}$, such that $\alpha$ fixes each element of the center $Z(\mathcal{M})$ of $\mathcal{M}$.  Then let $\mathcal{A} = \mathcal{M}\rtimes_\alpha \mathbb{Z}_+$.   Then every $\mathcal{A}$ and $\mathcal{A}^*$-invariant subspace $\mathcal{K}$ of $L^2(M,\tau)$ has the form $\mathcal{K}=vH^2$ for a partial isometry $v$ in the commutant of right multiplication by $\mathcal{M}$ on $L^2(\mathcal{M},\tau)$.}   This follows from theorem \ref{theorem6.2} when $\tau$ is finite, and $p=2$.

McAsey, Muhly and Saito's result is a corollary of a result by Nazaki and Watatani in \cite{NW}.  \emph{Suppose $\mathcal{M}$ is a finite von Neumann algebra with trace $\tau$, $\mathcal{D}\subseteq\mathcal{M}$ and a faithful, normal, trace-preserving conditional expectation $\Phi:\mathcal{M}\rightarrow \mathcal{D}$.  We let $H^\infty$ be a maximal subdiagonal algebra with respect to $\Phi$, and suppose that $Z(\mathcal{D})\subseteq Z(\mathcal{M})$.  Then, if we let $\mathcal{K}$ be a $H^\infty$-invariant subspace of $L^2(\mathcal{M},\tau)$ such that $\mathcal{K}$ is of $H^\infty$-type I (in the sense defined in \cite{NW}), there exists a partial isometry $v$ in the commutant of right multiplcation by $\mathcal{M}$ such that $\mathcal{K}=vH^2$.}  Again, this follows from our result in the finite case when $p=2$.

Similarly, Saito in \cite{Sai2} proves another Beurling-like theorem
for a finite von Neumann algebra $\mathcal{M}$. \emph{Let a closed
$\mathcal{K}$ of $L^2(\mathcal{M},\tau)$ invariant under
$\mathcal{M}\rtimes_\alpha \mathbb{Z}_+$ such that there are no
subspaces of $\mathcal{K}$ with $\mathcal{M}\rtimes_\alpha
\mathbb{Z} \mathcal{K}\subseteq \mathcal{K}$ such that $\mathcal{K}$
has the form $\sum_{n=0}^{\infty} \oplus V_nH^2$ with $\{V_n\}$ a
family of partial isometries with $\{V_n v_n^*\}$ is mutually
orthonogal. }

We are also able to prove a Beurling-Blecher-Labuschagne-like theorem for the Schatten $p$-classes for $0<p<\infty$, as described at the beginning of this section, using theorems \ref{theorem4.6}, \ref{theorem5.4} and corollary \ref{corollary5.5}.

{\renewcommand{\thetheorem}{\ref{cor6.3}}
\begin{corollary}
Let $\mathcal H$ be a separable Hilbert space with  an orthonormal base $\{e_m\}_{m\in\mathbb Z}$.
 Let $H^\infty$ be the lower triangular subalgebra of $B(\mathcal H)$, i.e.
$$
H^\infty =\{ x\in B(\mathcal H) : \langle xe_m, e_n\rangle =0, \ \ \forall n<m\}.
$$
 Let $\mathcal D=H^\infty\cap (H^\infty)^*$ be the diagonal subalgebra of $B(\mathcal H)$.
\begin{enumerate}\item [(i)]  For each $0<p< \infty$, let $S^p(\mathcal H)$  be the Schatten $p$-class.  Assume that $\mathcal K  $ is a closed
subspace of $S^p(\mathcal H)$ such that $H^\infty \mathcal K
\subseteq \mathcal K$.
             Then there exist a   projection $q$ in $\mathcal D$  and    a    family $\{u_\lambda\}_{\lambda \in\Lambda}$ of partial isometries in
             $ B(\mathcal H)$  satisfying
            \begin{enumerate}
              \item [(a)]  $u_\lambda q=0$  for all $\lambda \in \Lambda$;
              \item  [(b)]  $u_\lambda u_\lambda^*\in  \mathcal D  $   and $u_\lambda u_\mu^*=0$
for all $\lambda, \mu\in \Lambda$ with $\lambda\ne \mu$;
                \item [(c)] $\mathcal K=(S^p(\mathcal H)q) \oplus^{row}(  \oplus^{row}_{\lambda\in\Lambda } H^p u_\lambda).$
            \end{enumerate}
\item [(ii)]   Assume that $\mathcal K  $ is a weak $*$-closed
subspace of $ B(\mathcal H) $ such that $H^\infty \mathcal K
\subseteq \mathcal K$.
             Then there  exist a projection $q$ in $\mathcal D$  and   a    family $\{u_\lambda\}_{\lambda \in\Lambda}$ of partial isometries in
             $ B(\mathcal H)$  satisfying
            \begin{enumerate}
             \item [(a)]  $u_\lambda q=0$  for all $\lambda \in \Lambda$;
              \item  [(b)]  $u_\lambda u_\lambda^*\in \mathcal D$ and $u_\lambda u_\mu^*=0$
for all $\lambda, \mu\in \Lambda$ with $\lambda\ne \mu$;
                \item [(c)] $\mathcal K= ( B(\mathcal H)q )\oplus^{row}( \oplus^{row}_{\lambda\in\Lambda } H^\infty u_\lambda).$
            \end{enumerate}
\end{enumerate}

\end{corollary}
}

However, if we have that this projection $q$ in $\mathcal{D}$ has the characteristic that $S^p(\mathcal{H})q\subseteq H^p$, then we can prove corollary \ref{cor6.4}, and therefore fully characterize a $H^\infty$-invariant subspace $\mathcal{K}\subseteq H^p$ when $0<p\leq \infty$ and $\mathcal{H}$ is a separable Hilbert space with an orthonormal base.

{\renewcommand{\thetheorem}{\ref{cor6.4}}
\begin{corollary}
Let $\mathcal H$ be a separable Hilbert space with  an orthonormal base  $\{e_m\}_{m\in\mathbb Z}$.
 Let $H^\infty$ be the lower triangular subalgebra of $B(\mathcal H)$, i.e.
$$
H^\infty =\{ x\in B(\mathcal H) : \langle xe_m, e_n\rangle =0, \ \ \forall n<m\}.
$$ Let $\mathcal D=H^\infty\cap (H^\infty)^*$ be the diagonal subalgebra of $B(\mathcal H)$.
\begin{enumerate}\item [(i)]  For each $0<p< \infty$, if $\mathcal K  $ is a closed
subspace of $H^p$ such that $H^\infty \mathcal K
\subseteq \mathcal K$,
             then there exists   a    family $\{u_\lambda\}_{\lambda \in\Lambda}$ of partial isometries in
             $ H^\infty$  satisfying
            \begin{enumerate}
              \item  [(a)]  $u_\lambda u_\lambda^*\in  \mathcal D  $   and $u_\lambda u_\mu^*=0$
for all $\lambda, \mu\in \Lambda$ with $\lambda\ne \mu$;
                \item [(b)] $\mathcal K=  \oplus^{row}_{\lambda\in\Lambda } H^p u_\lambda .$
            \end{enumerate}
\item [(ii)]   Assume that $\mathcal K  $ is a weak $*$-closed
subspace of $ H^\infty $ such that $H^\infty \mathcal K
\subseteq \mathcal K$.
             Then there  exists     a    family $\{u_\lambda\}_{\lambda \in\Lambda}$ of partial isometries in
             $ H^\infty$  satisfying
            \begin{enumerate}
              \item  [(a)]  $u_\lambda u_\lambda^*\in \mathcal D$ and $u_\lambda u_\mu^*=0$
for all $\lambda, \mu\in \Lambda$ with $\lambda\ne \mu$;
                \item [(b)] $\mathcal K=  \oplus^{row}_{\lambda\in\Lambda } H^\infty u_\lambda .$
            \end{enumerate}
\end{enumerate}

\end{corollary}
}

Therefore, we are able to answer the question given in problem \ref{prob1.1} and fully characterize an $\mathcal{A}$-invariant subspace of a Schatten $p$-class: given a subpace $\mathcal{K}\subseteq S^p(\mathcal{H})$ such that $\mathcal{AK}\subseteq\mathcal{K}$, we have that $\mathcal{K}= (S^p(\mathcal{H})q) \oplus^{row}_{\lambda\in\Lambda } H^p u_\lambda$ when $0<p<\infty$, and $\mathcal{K}=(B(\mathcal{H})q)\oplus^{row}(\oplus^{row}_{\lambda\in\Lambda} H^\infty u_\lambda)$ when $p=\infty$.





The outline of the paper is as follows.   In section 2, we discuss
preliminary definitions and notations.  In section 3, we prove the
Beurling-Blecher-Labuschagne theorem for $L^p(\mathcal{M},\tau)$
when $p=2$, and extend this theorem in section 4 to the case when
$1\leq p\leq \infty$.  We provide several more preliminaries and
further extend the Beurling-Blecher-Labuschagne theorem to
$L^p(M,\tau)$ for $0<p<1$ in section 5.  Finally, in section 6, we
discuss some applications of our results on invariant subspaces for
analytic crossed products and discuss the results for the Schatten $p$-class.

\section{Preliminaries and Notation}
        In this section we give some preliminary definitions and results for non-commutative $L^p$ spaces for a von Neumann algebra with a tracial weight.  We then discuss Arveson's  non-commutative   Hardy space.

\subsection{Weak $*$-topology} Let $\mathcal M$ be a von Neumann
algebra with a predual $\mathcal M_\sharp$. Recall that the weak
$*$-topology   $\sigma(\mathcal M,\mathcal M_\sharp)$  on $\mathcal
M$ is a topology on $\mathcal M$ induced from the predual space
$\mathcal M_\sharp$. The following known result   (for example, see
Theorem 1.7.8 in \cite{Sakai}) is useful.

\begin{lemma}\label{Lemma2.1}
Let $\mathcal M$ be a von Neumann algebra. If
$\{e_\lambda\}_{\lambda\in\Lambda}$ is a net of projections in
$\mathcal M$ such that $e_\lambda\rightarrow I$ in weak
$*$-topology, then $e_\lambda x  \rightarrow x$, $ x
e_\lambda\rightarrow x$ and $e_\lambda x e_\lambda\rightarrow x$ in
weak $*$-topology for all $x$ in $\mathcal M$.
\end{lemma}

    \subsection{ Semifinite von Neumann Algebras}

    We begin with a description of a semifinite von Neumann algebra.

    Let $\mathcal{M}$ be a von Neumann algebra, and let $\mathcal{M}^+$ be the positive part of $\mathcal{M}$.  Recall that a mapping $\tau:\mathcal{M}^+\rightarrow [0,\infty]$ is a \textit{tracial weight} on $\mathcal{M}$ if
                \begin{enumerate}
                    \item $\tau(x+y)=\tau(x)+\tau(y)$ for $x,y \in \mathcal{M}^+$
                    \item $\tau(ax)=a\tau(x)$ for $x\in \mathcal{M}^+$ and $a\in[0,\infty]$
                    \item $\tau(xx^*)=\tau(x^*x)$ for every $x\in \mathcal{M}$.
                \end{enumerate}
    A tracial weight $\tau$ is called \textit{normal} if $\tau: \mathcal{M}^+\rightarrow \mathbb{C}$ is continuous with respect to the weak $*$-topology.
     $\tau$ is \textit{faithful} if for every $a\in \mathcal{M}^+$, $\tau(a^* a)=0$ implies $a=0$.  $\tau$ is said to be \textit{finite} if $\tau(I)<\infty$, and \textit{semifinite} if for any $x\in \mathcal{M}^+$, $x\neq 0$, there is a $y\in \mathcal{M}^+$, $y\neq 0$ such that $\tau(y)<\infty$ and $y\leq x$.  A von Neumann algebra $\mathcal{M}$ is called \textit{semifinite} if a faithful, normal semifinite $\tau$ exists.\\

The following lemma is well known.
\begin{lemma}\label{Lemma2.2}
Let $\mathcal M$ be a von Neumann algebra with a semifinite,
faithful, normal, tracial weight $\tau$. Then the following are
true.
\begin{enumerate}
\item There exists a family $\{e_j\}_{j\in J}$ of
orthogonal projections in $\mathcal M$ such that (i) $\sum_j e_j$
converges to $I$ in weak $*$-topology and (ii) $\tau(e_j)<\infty$
for each $ {j}\in J $.

\item There exists a net $\{e_\lambda\}_{\lambda\in\Lambda}$  of
projections in $\mathcal M$ such that (i) $e_\lambda\rightarrow I$
in weak $*$-topology and (ii) $\tau(e_\lambda)<\infty$ for each $
{\lambda}\in \Lambda$.

\end{enumerate}
\end{lemma}
\begin{proof} It is not hard to see that
(2) follows   from (1). For the purpose of completeness, we sketch
the  proof of (1) here. Actually, we need only to show that every
nonzero projection $e$ in $\mathcal M$ contains a nonzero
subprojection $\tilde e$ such that $ \tau(\tilde e)<\infty$.  Then
the rest follows directly from Zorn's lemma.

Let $e $ be a nonzero projection in $\mathcal M$. Since $\tau$ is
semifinite, then  is a $y\in \mathcal{M}^+$, $y\neq 0$ such that
$\tau(y)<\infty$ and $y\leq f$. Therefore, there exist a positive
number $\lambda>0$ and a nonzero spectral projection $\tilde e$ of
$y$ in $\mathcal M$ such that $\lambda \tilde e\le y$. Hence $\tilde
e$ is a non-zero subprojection of $e$ such that $ \tau (\tilde
e)<\infty$.
\end{proof}

\subsection{$L^p$-spaces of semifinite von Neumann algebras} Let
$\mathcal M$ be a von Neumann algebra with a semifinite, faithful,
normal, tracial weight $\tau$. We let
$$\mathcal{I}=span\{MeM:e=e^*=e^2 \in M \text{ with } \tau(e)<
\infty\}$$ be the set of elementary operators in $\mathcal M$ (see
Definition 3.1 in \cite{Se}). Then $\mathcal I$ is a two-sided ideal
of $\mathcal M$.

For each $0<p<\infty$, we define a mapping $\|\cdot\|_p: \mathcal
I\rightarrow [0,\infty)$ as follows
        $$\|x\|_p=(\tau(|x|)^p)^{\frac{1}{p}} \qquad \text{ for every } x\in\mathcal{I}.$$
     It is a highly trivial fact that $\|\cdot\|_p$ is a norm  on $\mathcal I$
     for $1\le p<\infty$, and a $p$-norm on $\mathcal I$   for $0< p< 1$. (see Theorem 4.9 in \cite{Fa}) \\ 

\begin{definition}\label{Definition2.3}
Let $\mathcal M$ be a von Neumann algebra with a semifinite,
faithful, normal, tracial weight $\tau$, and
$\mathcal{I}=span\{MeM:e=e^*=e^2 \in M \text{ with } \tau(e)<
\infty\} $ be the set of elementary operators in $\mathcal M$. We
define $L^p(\mathcal M,\tau)$, for $0<p<\infty$, to be the
completion of $\mathcal I$ under $\|\cdot \|_p$, i.e.
$$
L^p(\mathcal M,\tau)  =\overline{\mathcal I}^{\|\cdot \|_p}.
$$

As usual, we  let $L^\infty(\mathcal M, \tau)$ be $\mathcal M$.

\end{definition}

\begin{notation}If $S$ is a subset of $ L^p(\mathcal M,\tau)$ with
$0<p<\infty$, we will denote by $[S]_p$ the closure of $S$ in $
L^p(\mathcal M,\tau)$. If $S$ is a subset of $ \mathcal M $, we will
denote by $\overline{S}^{w*}$ the closure of $S$ in $\mathcal M $ under
the weak $*$-topology.
\end{notation}

The following two lemmas are well known.
        \begin{lemma}\label{Lemma2.5}Let $\mathcal M$ be a von Neumann algebra with a semifinite,
faithful, normal, tracial weight $\tau$. The following are true.
\begin{enumerate}
\item (H\"{o}lder's Inequality) For $0<p,q,r\le \infty$ with $1/p+1/q=1/r$, we have  $$\|xy\|_r\le \|x\|_{p}\|y\|_{q} \qquad \text{  for
all $x\in L^p(\mathcal M,\tau) $ and $y\in L^q(\mathcal M,\tau)$}.$$

 \item For each $0<r\le \infty$, we have           $\|axb\|_r\leq\|a\| \|x\|_r \|b\|$ for $x\in L^r(\mathcal M,\tau) $ and $a,b\in
 \mathcal{M}$. Therefore, $L^r(\mathcal M,\tau) $ is an $\mathcal M$
 bi-module for each $0<r\le \infty$.
 \item (Duality) For any $1\le p<\infty$ and $1<q\le \infty$ with $1/p+1/q=1$, we have
 $$
 (L^p(\mathcal M,\tau))^\sharp = L^q(\mathcal M,\tau) \qquad (\text{isometrically}),
 $$ where the duality between $L^p(\mathcal M,\tau)$ and $L^q(\mathcal M,\tau)$ is given by
$\langle x, y\rangle =\tau(xy).$  Thus, $L^1(\mathcal M,\tau)$ is the predual of $\mathcal M$.
 \end{enumerate}       \end{lemma}
 \begin{proof}
See   \cite {Fa}.
 \end{proof}

 \begin{lemma}\label{Lemma2.6.2}Let $\mathcal M$ be a von Neumann algebra with a semifinite,
faithful, normal, tracial weight $\tau$ and $0<p<\infty$. If
$\{e_\lambda\}_{\lambda\in\Lambda}$ is a net of projections in
$\mathcal M$ such that such that   $e_\lambda\rightarrow I$ in the weak
$*$-topology, then   for
every  $ x\in L^p(\mathcal M,\tau)$
$$\lim_\lambda \|e_\lambda x   -x\|_p=0; \ \ \lim_\lambda \|  x e_\lambda -x\|_p=0; \ \ \text{ and } \ \
\lim_\lambda \|e_\lambda x e_\lambda -x\|_p=0.
$$
        \end{lemma}
\begin{proof} For the purpose of completeness, we include a proof here.  Notice
that the set $\mathcal I$ of elementary operators of $\mathcal M$ is
dense in $L^p(\mathcal M,\tau)$ (see Definition
\ref{Definition2.3}) and
$$ \|e_\lambda x e_\lambda -x\|_p= \|e_\lambda (x e_\lambda -  x) + e_\lambda x -x\|_p, \ \ \forall \ x\in  L^p(\mathcal M,\tau)
.$$  Because of Lemma \ref{Lemma2.5}, it suffices to show that,  for all $a, b \in \mathcal M$
and a projection $f$ in $\mathcal M$ with $\tau (f)<\infty$, we have
$ \lim_\lambda \|e_\lambda  (afb)
  -afb\|_p=0  $ and $ \lim_\lambda \|  (afb)
e_\lambda -afb\|_p=0. $

Assume that $0<p<2$. Let $q$ be a positive number such that
$1/p=1/2+1/q$. We have
\begin{align}
 \|e_\lambda  (afb) -afb\|_p &=  \|(e_\lambda-I)  af b\|_p \notag \\
 &\le  \|b\| \|(e_\lambda-I) af\|_p \tag{by Lemma \ref{Lemma2.5}}\\
 &\le \|b\| \|(e_\lambda-I) af\|_2 \|f\|_q \tag{ by H\"{o}lder's Inequality
 }\\
 &= \|b\|\|f\|_q \left ( \tau (fa^*(I-e_\lambda)af) \right
 )^{1/2}\notag \\
 &= \|b\|\|f\|_q \left ( \tau ( (I-e_\lambda)afa^*) \right
 )^{1/2}.\notag
\end{align}
Observe that   $e_\lambda\rightarrow I$ in weak $*$-topology and $
afa^*\in L^1(\mathcal M,\tau)$ as $\tau(f)<\infty$. We have that
\begin{align}
 \lim_\lambda  \tau ( (I-e_\lambda)afa^*)
   =0. \label{equation2.3.1}
\end{align} It follows that $ \lim_\lambda \|e_\lambda  (afb)
  -afb\|_p=0.  $ Furthermore, we have that $ \lim_\lambda \|  (afb)
e_\lambda -afb\|_p= \lim_\lambda \|   e_\lambda b^*fa^* -
b^*fa^*\|_p=0, $ for $0<p<2$.

Assume that $2\le p<\infty$. We have
\begin{align}
\|e_\lambda  (afb) -afb\|_p &\le  \|b\| \|(e_\lambda-I) af\|_p  \tag{by Lemma \ref{Lemma2.5}}\\
&= \|b\| \left (\tau( (fa^*(I-e_\lambda)af)^{p/2})  \right
)^{1/p}\notag\\
&\le \|b\| \|(fa^*(I-e_\lambda)af)^{\frac p 2 -1}\|^{1/p} \left
(\tau  (fa^*(I-e_\lambda)af)  \right )^{1/p} \tag{by the property of
$\tau$}
\end{align}
Note from Equation (\ref{equation2.3.1}) that
$$
 \lim_\lambda \tau  (fa^*(I-e_\lambda)af) =\lim_\lambda \tau ( (I-e_\lambda)afa^*)
 =0.$$
 We have that $ \lim_\lambda \|e_\lambda  (afb)
  -afb\|_p=0.  $ Furthermore, we have that $ \lim_\lambda \|  (afb)
e_\lambda -afb\|_p= \lim_\lambda \|   e_\lambda b^*fa^* -
b^*fa^*\|_p=0, $ for $2\le p<\infty$.

This ends the proof of the lemma.
\end{proof}

    \subsection{Arveson's Non-Commutative Hardy Space}
    In this subsection, we will recall Arveson's definition of non-commutative Hardy spaces. Assume that $\mathcal{M}$ is a von Neumann algebra
     with a semifinite, faithful, normal tracial weight $\tau$. Assume $\mathcal{A}\subseteq \mathcal{M}$ is a weak*-closed subalgebra of $\mathcal M$,
     and let $\mathcal{D}=\mathcal{A}\cap\mathcal{A}^*$.  Assume that $\Phi:\mathcal{M}\rightarrow\mathcal{D}$ is faithful,
     normal
      conditional expectation  from $\mathcal M$ onto $\mathcal D$. 

    \newtheorem{hardyspace}[theorem]{Definition}
        \begin{hardyspace} \label{Definition2.7}
          $\mathcal{A}$ is a called a semifinite subdigonal subalgebra, or a semifinite non-commutative Hardy space,  with respect to $(\mathcal M, \Phi)$ if
                \begin{enumerate}
                   \item The restriction of $\tau$ on
        $\mathcal{D}=\mathcal{A}\cap\mathcal{A}^*$ is semifinite.
                    \item $\Phi(xy)=\Phi(x)\Phi(y)$ for every $x,y\in \mathcal{A}$.
                    \item $\mathcal{A}+\mathcal{A}^*$ is weak* dense in $M$.
                    \item $\tau(\Phi(x))=\tau(x)$ for every positive operator $x$ in $M$.
                \end{enumerate}
        In this case, $\mathcal A$ will also be denoted by
        $H^\infty$. Furthermore, we denote $[\mathcal
        A\cap L^p(\mathcal M,\tau)]_p$ by $H^p$ for each $0<p<\infty$.
        \end {hardyspace}
\begin{remark}\label{Remark2.8}
It was shown in \cite{Xu} and \cite {Ji} that such a subalgebra
$H^\infty$   with respect to $(\mathcal M,\Phi)$ is maximal among
semifinite subdigonal subalgebras satisfying (1), (2), (3) and (4). From this fact,
it follows that
$$
H^\infty= \{  a\in \mathcal M  \ : \ \Phi (xay)=0, \ \ \forall \
x\in H^\infty, y\in H^\infty\cap Ker(\Phi)\}.
$$
\end{remark}

\begin{remark}\label{Remark2.10} Following notation from Definition
\ref{Definition2.7}, we know that the conditional expectation
$\Phi:\mathcal M\rightarrow
  \mathcal D$ can be extended to a projection from $L^p(\mathcal
  M,\tau)$ onto $L^p(\mathcal D,\tau)$ for each $1\le p<\infty$ (see
  Proposition 2.3 in \cite{Xu} or \cite{Be}). Such
  extended projection will still be denoted by $\Phi$.  Moreover,
$$
\Phi(axb)=a\Phi(x)b, \qquad \forall \ a, b\in \mathcal D, \ x\in
L^p(\mathcal
  M,\tau), \ 1\le p<\infty.
$$ \end{remark}

\begin{notation} We will let $H^\infty_0= \ker(\Phi)\cap H^\infty$ and $H^p_0
=\ker(\Phi)\cap H^p$.
\end{notation}

The next result follows directly from Definition
\ref{Definition2.7} and can be found in Lemma 3.1 of \cite{Be}.
        \begin{lemma} \label{Lemma2.10}
            If $e$ is a projection in $ \mathcal{D}=H^\infty\cap (H^\infty)^*$  with $0<\tau(e)<\infty$, then $eH^\infty e$ is a finite subdiagonal subalgebra  of $e\mathcal{M}e$, and $[eH^\infty e]_p=eH^pe$.
        \end{lemma}

 %



We will need the following technical lemma in the later sections.

        \begin{lemma}\label{lemma2.12}
            Suppose $\mathcal{M}$ is a   von Neumann algebra with a semifinite, faithful, normal, tracial weight $\tau$. Let $H^\infty$ be a semifinite subdiagonal subalgebra in $\mathcal M$ in the sense of Definition \ref{Definition2.7} (namely, the restriction of $\tau$ on
        $\mathcal{D}=H^\infty\cap (H^\infty)^*$ is semifinite).

         Then for every $x\in L^p(M,\tau)$ with $0< p <\infty$  and for every $e\in\mathcal{D}$ with $0<\tau(e)<\infty$, there exist  an $h_1, h_3\in eH^\infty e$ and an $h_2, h_4 \in eH^p e$ such that:
            \begin{enumerate}
                \item[(i)] $h_1 h_2=e=h_2 h_1$ and $h_3h_4=h_4h_3=e$
                \item[(ii)] $h_1 ex$ and $xeh_3$ are in
                $\mathcal{M}$.
            \end{enumerate}
        \end{lemma}

    \begin{proof}
        Let
            $$ex=\sqrt{exx^*e}u_1=|x^*e|u_1$$
       be  the polar decomposition of $ex$ in $L^p(\mathcal M,\tau)$,
        where $u_1$ is a partial isometry in $\mathcal M$ and $|x^*e|$ is a positive operator in $L^p(\mathcal M,\tau)$.
        It is not hard to see that  $|x^*e|\in eL^p(\mathcal{M},\tau)e=L^p(e\mathcal{M}e,\tau)$. Since $0<\tau(e)<\infty$,
        we know that  $e\mathcal{M}e$ is a finite von Neumann algebra with a faithful normal
        tracial state $\frac{1}{\tau(e)}\tau$. And, from Lemma \ref{Lemma2.10}, it follows that $eH^\infty e$ is a finite
        subdigonal subalgebra of $e\mathcal Me$ with $[eH^\infty e]_p=eH^pe$. Note that
         $|x^*e|\in L^p(e\mathcal{M}e,\tau)$ and $0<\tau(e)<\infty$. Then $w=(e+|x^*e|)^{-1}$ is an
         invertible operator in $e\mathcal Me$ with $w^{-1}\in L^p(e\mathcal{M}e,\frac{1}{\tau(e)}\tau)$.
          From Theorem 3.1 in \cite{BX},   there exist a unitary $v $ in $e\mathcal Me$,  an $h_1\in eH^\infty e$  and an  $h_2\in eH^p e$ such that (i) $h_1 h_2 =e= h_2 h_1$, and (ii$_a$) $w=vh_1$. Now from (ii$_a$) we have      (ii$_b$) $h_1|x^* e| = v^*w|x^* e|= v^* (e+|x^*e|)^{-1} |x^* e| \in e\mathcal{M}e\subseteq\mathcal{M}$. Hence, from (ii$_b$) and the fact that  $u_1$ is a partial isometry in $\mathcal M$, we obtain that (ii) $h_1ex=h_1  |x^* e|u_1 \in \mathcal{M}.$
    The proof for the existence of $h_3$ and $h_4$ is similar.   This ends the proof of the lemma.
    \end{proof}

The following lemma is also useful.
\begin{lemma}\label{lemma2.13}
Suppose $\mathcal{M}$ is a   von Neumann algebra with a semifinite,
faithful, normal, tracial weight $\tau$. Let $H^\infty$ be a
semifinite subdiagonal subalgebra with respect to  $(\mathcal
M,\Phi)$, where $\Phi$ is a faithful, normal conditional expectation
from $\mathcal M$ onto
        $\mathcal{D}=H^\infty\cap (H^\infty)^*$.

Then there exists        a net
            $\{e_\lambda\}_{\lambda\in\Lambda}$ of projections in
            $\mathcal D$ such that  such that
\begin{enumerate}
\item [(i)]            $e_\lambda \rightarrow I$ in the weak $*$-topology of $\mathcal M$ and
$\tau(e_\lambda)<\infty$ for each $\lambda\in\Lambda$.
\item [(ii)] We have,   for
every  $ x\in L^p(\mathcal M,\tau)$ with $0<p<\infty$,
$$\lim_\lambda \|e_\lambda x   -x\|_p=0; \ \ \lim_\lambda \|  x e_\lambda -x\|_p=0; \ \ \text{ and } \ \
\lim_\lambda \|e_\lambda x e_\lambda -x\|_p=0.
$$
\end{enumerate}
\end{lemma}
\begin{proof}
   Since $H^\infty$ is a semifinite subdiagonal subalgebra
            of $\mathcal M$, the restriction of $\tau$ on  $\mathcal{D}  $ is semifinite. By Lemma
            \ref{Lemma2.2}, there exists a net
            $\{e_\lambda\}_{\lambda\in\Lambda}$ of projections in
            $\mathcal D$ such that    $e_\lambda \rightarrow I$ in the weak $*$-topology of $\mathcal D$ and $\tau(e_\lambda)<\infty$ for each $\lambda\in\Lambda$.
Thus
$$
\lim_{\lambda} |\tau(e_\lambda z -z)|=0, \qquad \forall \ z\in
L^1(\mathcal D,\tau).
$$ For each $y\in L^1(\mathcal M,\tau)$, we have
$$
\lim_{\lambda} |\tau(e_\lambda y -y)|=  \lim_{\lambda}
|\tau(\Phi(e_\lambda y -y))| = \lim_{\lambda} |\tau(e_\lambda\Phi(
y) -\Phi( y))| =0
.
$$ i.e.
\begin{enumerate}
\item [(i)]            $e_\lambda \rightarrow I$ in the weak $*$-topology of $\mathcal M$ and
$\tau(e_\lambda)<\infty$ for each $\lambda\in\Lambda$.
\end{enumerate}
From (i) and Lemma \ref{Lemma2.6.2}, we induce that
\begin{enumerate}
\item [(ii)] For
every  $ x\in L^p(\mathcal M,\tau)$,
$$\lim_\lambda \|e_\lambda x   -x\|_p=0; \ \ \lim_\lambda \|  x e_\lambda -x\|_p=0; \ \ \text{ and } \ \
\lim_\lambda \|e_\lambda x e_\lambda -x\|_p=0.
$$
\end{enumerate} This ends the proof of the lemma.
\end{proof}

Now we recall the following definition for the row sum of subspaces in $L^p(\mathcal
M,\tau)$ for $0<p\le \infty$ as follows.
\begin{definition}\label{definition2.12}
  Let $\mathcal M$ be a  von Neumann algebra with a semifinite, normal
faithful, tracial weight $\tau$ and $0<p<\infty$. Let $X$ be a closed subspace of $L^p(\mathcal
M,\tau)$. Then $X$ is called an internal row sum of
closed subspaces $\{X_i\}_{i\in \mathcal I}$ of $L^p(\mathcal
M,\tau)$, denoted by
$
X=\bigoplus^{row}%
_{i\in \mathcal I} X_i
$,  if
\begin{enumerate}
\item $X_jX_i^*=\{0\}$ for all distinct $i,j\in \mathcal I$; and
\item the linear span of $\{X_i\ :  \ i\in \mathcal I\}$ is
 dense  in $X$, i.e. $X=[span  \{X_i\ :  \ i\in \mathcal I\}]_p.$
\end{enumerate}
\end{definition}

\begin{definition}\label{definition2.15}
  Let $\mathcal M$ be a  von Neumann algebra. Let $X$ be a weak $*$-closed
subspace of $\mathcal M$. Then $X$ is called an internal row sum
of a family of weak*-closed subspaces $\{X_i\}_{i\in \mathcal I}$ of $\mathcal
M$, denoted by
$
X=\bigoplus^{row}%
_{i\in \mathcal I} X_i
$,  if
\begin{enumerate}
\item $X_j X_i^*=\{0\}$ for all distinct $i,j\in \mathcal I$; and
\item the linear span of $\{X_i\ :  \ i\in \mathcal I\}$ is
weak*-dense in $X$, i.e. $X=\overline{ span  \{X_i\ :  \ i\in \mathcal I\} }^{w*}$.
\end{enumerate}
\end{definition}

\section{Beurling-Blecher-Labuschagne Theorem for Semifinite Hardy Spaces, p=2}

\subsection{Main Result} In this section, we will prove a Beurling-Blecher-Labuschagne type theorem for semifinite non-commutative Hardy
spaces.
\begin{theorem}\label{theorem3.1}
Let $\mathcal M$ be a von Neumann algebra with a faithful, normal,
semifinite tracial weight $\tau$, and $H^\infty$ be  a
weak$^*$-closed subalgebra   of $\mathcal M$. Let $\mathcal
D=H^\infty\cap (H^\infty)^*$ be a von Neumann subalgebra of
$\mathcal M$, and $\Phi:\mathcal M\rightarrow \mathcal D$ be  a
faithful normal condition expectation.

Assume that $H^\infty$ is a semifinite subdigonal subalgebra with
respect to  $(\mathcal M,\Phi)$ (see Definition
\ref{Definition2.7}). Let $\mathcal K$ be a closed subspace of
$L^2(\mathcal M,\tau)$ satisfying $H^\infty\mathcal K\subseteq
\mathcal K$. Then there exist a closed subspace $Y$ of $L^2(\mathcal
M,\tau)$ and a family $\{u_\lambda\}_{\lambda\in\Lambda}$ of partial
isometries in $\mathcal M$, satisfying
\begin{enumerate}

\item  [(i)] $u_\lambda Y^*=0$ for all ${\lambda\in\Lambda}$.
\item  [(ii)] $u_\lambda u_\lambda^*\in \mathcal D$ and $u_\lambda u_\mu^*=0$
for all $\lambda, \mu\in \Lambda$ with $\lambda\ne \mu$.
\item  [(iii)] $Y=[H^\infty_0Y]_2$, where $H^\infty_0=H^\infty\cap
ker(\Phi)$.
\item [(iv)] $\mathcal K=Y\oplus \left ( \oplus_{\lambda\in\Lambda} H^2u_\lambda \right
)$
\end{enumerate}

\end{theorem}

The proof of this   result uses a similar idea as the one in
\cite{BL2} for finite von Neumann algebras. We will  modify the
argument in \cite{BL2} to prove preceding result for the case of
semifinite von Neumann algebras. First, we present a series of
technical lemmas.

\subsection{Some lemmas} Following the notation above, we let $\mathcal M$ be a von Neumann algebra with a faithful, normal,
semifinite tracial weight $\tau$ and $H^\infty$ be a semifinite
subdigonal subalgebra   of $\mathcal M$. Let $\mathcal
D=H^\infty\cap (H^\infty)^*$ be a von Neumann subalgebra of
$\mathcal M$ and $\Phi:\mathcal M\rightarrow \mathcal D$   a
faithful normal conditional expectation. From Remark \ref{Remark2.10},
we know that $\Phi$ can be extended to a positive contraction from
$L^p(\mathcal
  M,\tau)$ onto $L^p(\mathcal D,\tau)$ for each $1\le p<\infty$ such
  that
$$
\Phi(axb)=a\Phi(x)b, \qquad \forall \ a, b\in \mathcal D, \ x\in
L^p(\mathcal
  M,\tau), \ 1\le p<\infty.
$$

We find the following   lemma is useful.
\begin{lemma}\label{lemma3.2}
Let $\mathcal M$ be a von Neumann algebra with a faithful, normal,
semifinite tracial weight $\tau$, and $H^\infty$  be  a semifinite
subdigonal subalgebra   of $\mathcal M$. If $x$ in $L^1(\mathcal
M,\tau)$ satisfies
$$
\tau(xz)=0 \qquad \text { for all } z\in H^\infty + (H^\infty)^*,
$$ then $x=0$.
\end{lemma}
\begin{proof}Assume that $x$ in $L^1(\mathcal
M,\tau)$ satisfies
$$
\tau(xz)=0 \qquad \text { for all } z\in H^\infty + (H^\infty)^*.
$$
Since $H^\infty$ is  a semifinite subdigonal subalgebra   of $\mathcal
M$, $H^\infty + (H^\infty)^*$ is weak$^*$-dense in $\mathcal M$.
From the fact that $x \in L^1(\mathcal M,\tau)$, we get that
$$
\tau(xz)=0 \qquad \text { for all } z\in \mathcal M.
$$ As $L^1(\mathcal M,\tau)$ is the predual space of $\mathcal M$,
we must have $x=0$.
\end{proof}

\begin{lemma}\label{lemma3.4}  Let $\mathcal K$ be a closed subspace of
$L^2(\mathcal M,\tau)$ satisfying $H^\infty\mathcal K\subseteq
\mathcal K$.  Let $$X=\mathcal K\ominus[H_0 ^\infty \mathcal
K]_2\subseteq \mathcal K\subseteq L^2(\mathcal{M},\tau).$$ Then the
following are true.
\begin{enumerate}
  \item [(i)] $XX^*\subseteq L^1(\mathcal{D},\tau)$.
  \item  [(ii)]  $X$ is a left $\mathcal{D}$-module, i.e. for every $d\in \mathcal{D}$ and $x\in X$, we have $dx\in X$.
  \item  [(iii)] Let $x$ be an element in $X$ and $x=hu$ be the
       polar decomposition of $x$ in $L^2(\mathcal M, \tau)$, where
       $u$ is a partial isometry in $\mathcal M$ and $h=|x^*|\in L^2(\mathcal M,
       \tau)$. Then  \begin{enumerate}
       \item [(a)] $h\in L^2(\mathcal D,\tau)$ and  $uu^*\in\mathcal D$;
 \item [ (b)] $[\mathcal D x ]_2 = L^2(\mathcal D,\tau) u;$   \item [(c)] $[H^\infty
  x ]_2=H^2u$. In particular, $H^2u\subseteq X$.
   \end{enumerate}
  \item  [(iv)] There exists a family $\{u_\lambda\}_{\lambda\in\Lambda}$ of partial isometries in
  $\mathcal M$ such that \begin{enumerate}
    \item [(a)]  $X= \oplus_{\lambda\in\Lambda} H^2u_\lambda;$ \item [(b)] $u_\lambda u_\lambda^*$ is a projection in $ \mathcal D$;  and \item [(c)] $u_\lambda u_\mu^*=0$
for all $\lambda, \mu\in \Lambda$ with $\lambda\ne \mu$.
\end{enumerate}
\end{enumerate}
\end{lemma}

    \begin{proof}
        (i): It is equivalent  to show  that for every $x,y\in X$, $yx^*\in L^1
        (\mathcal{D},\tau)$.

        Assume that
          $x,y\in X\subseteq  L^2(\mathcal{M},\tau)$. Thus  $yx^*\in L^1(\mathcal{M},\tau)$.
        Recall  $\Phi:L^1(\mathcal{M},\tau)\rightarrow
        L^1(\mathcal{D},\tau)$ is a positive contraction such that
$$
\Phi(d_1 a d_2)=d_1\Phi(a)d_2,\qquad  \forall \ d_1,d_2\in\mathcal D
\ \text { and } \ a\in L^1(\mathcal{M},\tau),
$$ and thus
\begin{align}
\Phi(d  a  )=d \Phi(a),\qquad  \forall \ d \in L^1(\mathcal D,\tau)
\ \text { and } \ a\in  \mathcal{M}. \label{equation3.1}
\end{align}
 Thus, to prove that   $yx^*\in L^1
        (\mathcal{D},\tau)$, it  is enough to show that $yx^*-\Phi(yx^*)=0$.
By Lemma \ref{lemma3.2} and the fact that $yx^*-\Phi(yx^*)\in
L^1(\mathcal{M},\tau)$, we need only to prove that
$$\text{$\tau([yx^*-\Phi(yx^*)]z)=0$    \ \   for every $z\in
H^\infty+(H^\infty)^*$.}$$ We will proceed the proof according to
the cases (1) $z\in H_0 ^{\infty}$, (2) $z\in\mathcal{D}$, and (3)
$z\in (H_0 ^\infty)^*$.

\noindent               Case (1):
                Let $z\in H_0 ^{\infty}$.  Then
                    \begin{align}
                         \tau([yx^*-\Phi(yx^*)]z)
                        &=\tau(yx^* z)-\tau(\Phi(yx^*)z) \notag \\
                        &=\tau(yx^*z)-\tau(\Phi(\Phi(yx^*)z)) \tag{  $\Phi$ is trace preserving} \\
                        &=\tau(zyx^*)-\tau(\Phi(yx^*)\Phi(z))\tag{    by equation \ref{equation3.1}} \\
                        &=0 \tag{as $x,y$  are in $X$ and  $ z$ is in
                        $H^\infty_0$}
                    \end{align}

  \noindent            Case (2):
            Let $z\in\mathcal{D}$.  Then
                \begin{align}
                    \tau([yx^*-\Phi(yx^*)]z)
                    &=\tau(\Phi([yx^*-\Phi(yx^*)]z)) \tag{  $\Phi$ is trace preserving}\\
                    &=\tau([\Phi(yx^*)-\Phi(yx^*)]z) \notag\\
                    & =0.\tag{as $x,y$  are in $X$ and  $ z$ is in
                        $H^\infty_0$}
                \end{align}

 \noindent             Case (3):
            Let $z\in (H_0 ^\infty)^*$.  Then
                \begin{align}  \tau([(yx^*)-\Phi(yx^*)]z)
                    &=\tau(yx^*z)-\tau(\Phi(yx^*)z)\notag\\
                    &=\tau(y(z^*x)^*)-\tau(\Phi(\Phi(yx^*)z)) \tag{  $\Phi$ is trace preserving}\\
                    &=\tau(y(z^*x)^*)-\tau(\Phi(yx^*)\Phi(z))\tag{    by equation \ref{equation3.1}} \\&=0\tag{as $x,y$  are in $X$ and  $ z$ is in
                        $H^\infty_0$}
                \end{align}
      This ends the proof of part (i).

      (ii): Let $d\in\mathcal D$ and $x\in X\subseteq \mathcal K $. Since $H^\infty \mathcal K\subseteq \mathcal K$, we have $dx\in\mathcal
      K$. Now, for $h_0 \in H_0^\infty$ and $k\in K$,
            $$\tau(h_0 k(dx)^*)=\tau(h_0 kx^* d^*)=\tau(d^*h_0 kx^*) =0,$$
as $d^*h_0\in H^\infty_0$, and $x\in X=\mathcal K\ominus[H_0 ^\infty
\mathcal K]_2.$ Hence $dx\perp [H_0^\infty K]_2$. Thus $dx\in X$ and
$X$ is a left $\mathcal D$-module.

       (iii): Assume $x$ is an element in $X$. Let $x=hu$ be the
       polar decomposition of $x$ in $L^2(\mathcal M, \tau)$, where
       $u$ is a partial isometry in $\mathcal M$ and $h=|x^*|\in L^2(\mathcal M,
       \tau)$. From the result in (i), we know that $h$ is
       in $L^2(\mathcal D,\tau)$. Therefore $uu^*$, as the range projection of $h$, is
       in $\mathcal D$. This shows that (a) is true.

      From (a), it follows that $[    L^2(\mathcal D,\tau)uu^*]_2=    L^2(\mathcal D,\tau)(uu^*).$
      Observe that $uu^*$ is the range projection   of $h$. Therefore, we have $[\mathcal Dh ]_2=
L^2(\mathcal D,\tau)(uu^*),$ whence
\begin{align}[\mathcal Dh ]_2u=   L^2(\mathcal D,\tau)(uu^*u)=  L^2(\mathcal D,\tau)u.\label{equation3.1.1}\end{align}
 We claim that
           $$[\mathcal{D}x  ]_2=[\mathcal{D}h  ]_2u.$$
In fact, let $\xi\in [\mathcal{D} h  ]_2$. There exists a sequence
$\{\xi_n\}_{n\in\mathbb N}$ in $\mathcal{D} h  $ such that $\xi_n
\rightarrow \xi$ in $||\cdot ||_2$-norm.
        Then we have that $\xi_n u  \rightarrow \xi u $ in $||\cdot ||_2$-norm.
        From the fact that $\xi_n u \in [\mathcal{D}h  u  ]_2$, we conclude that $\xi u  \in [\mathcal{D} h  u  ]_2$.
        Therefore, we have that \begin{align}[\mathcal{D}h ]_2 u  \subseteq[\mathcal{D}h  u ]_2
        =[\mathcal{D}x
        ]_2.\label{equation3.2}\end{align}
        Now  let $\xi \in [\mathcal{D}h  u  ]_2=[\mathcal{D}x
        ]_2$.  There exists  a sequence $\{d_n\}_{n\in\mathbb N}$
        in $\mathcal{D}  $
    such that $d_n h u \rightarrow \xi$ in $||\cdot ||_2$-norm.
            Let
         $\eta_n =  d_n h  u  u^*   =d_n h  \in  \mathcal{D}h . $
Then
        $ \eta_n   \rightarrow \xi u^*  \in [\mathcal{D}h ]_2$
     in $|| \cdot ||_2$-norm.  Thus
        $ \eta_n u    =d_n h  u   \rightarrow \xi u^*u$
   in $|| \cdot ||_2$-norm. Combining with the fact that $d_n h u \rightarrow \xi$, we
   get
        $ \xi=((\xi u^* )u) \in [\mathcal{D}h  ]_2 u
        .$ Or
\begin{align} [\mathcal{D}h  u ]_2=[\mathcal{D}x
       ]_2\subseteq [\mathcal{D}h  ]_2 u .\label{equation3.3}\end{align}
   From equation (\ref{equation3.2}), equation (\ref{equation3.3}) and equation (\ref{equation3.1.1}), we conclude that
   $$[\mathcal{D}x
    ]_2 =[\mathcal{D}h  ]_2u =   L^2(\mathcal D,\tau)  u
        .$$ This ends the proof of part (b).
 The proof of (c) is similar to  (b).

(iv) We may assume that $X\ne 0$. From the result in (iii) and
Zorn's lemma, we may assume that there exists a maximal family
$\{u_\lambda\}_{\lambda\in \Lambda}$ of nonzero partial isometries
in $\mathcal M$ with respect to
  \begin{enumerate} \item [(a$_1$)]  $H^2u_\lambda\subseteq X$ for each $\lambda\in\Lambda$; \item [(b)] $u_\lambda u_\lambda^*$ is a projection in $ \mathcal D$;  and \item [(c)] $u_\lambda u_\mu^*=0$
for all $\lambda, \mu\in \Lambda$ with $\lambda\ne \mu$.
\end{enumerate} We will  show  that \begin{enumerate}   \item [(a)]  $X=
\oplus_{\lambda\in\Lambda} H^2u_\lambda.$ \end{enumerate}

 In fact,
from  (a$_1$), we know that each $H^2u_\lambda\subseteq X$.
Combining with (c), we conclude that $\{H^2u_\lambda\}_{\lambda\in
\Lambda}$ is a family of orthogonal subspaces of $X$, whence
$\oplus_{\lambda\in\Lambda} H^2u_\lambda$ is a subspace of $X$.

Now assume that $X\ominus (\oplus_{\lambda\in\Lambda} H^2u_\lambda)$
is not equal to $0$. Pick a nonzero $x$ in $X\ominus
(\oplus_{\lambda\in\Lambda} H^2u_\lambda)$ and assume that $x=hu$ is
the
       polar decomposition of $x$ in $L^2(\mathcal M, \tau)$, where
       $u$ is a nonzero partial isometry in $\mathcal M$ and $h=|x^*|\in L^2(\mathcal M,
       \tau)$. It follows from the result proved in (iii)  that $H^2u \subseteq
       X$ and
       $uu^*$ is  in $\mathcal D$.

    By Lemma
            \ref{lemma2.13}, there exists a net
            $\{e_j\}_{j\in J}$ of projections in
            $\mathcal D$ such that  such that $e_j \rightarrow I$ in the weak $*$-topology  and $\tau(e_j)<\infty$ for each $j\in J$.

     Let $j\in J$. Then by the
       choice of $x$, we get that   $H^2 u_\lambda $ and $x$ are
       orthogonal. So,
       $$
       \tau(de_ju_\lambda x^*)=0, \qquad \forall \ d\in\mathcal D.
       $$
       From (i),   $e_ju_\lambda x^*$ is in   $L^1(\mathcal D,\tau)$.   By Lemma \ref{lemma3.2}  we conclude that  $e_ju_\lambda
       x^*=0$ for each $j\in J$.

    As $u_\lambda x^* \in L^2(\mathcal M,\tau)$, $\lim_{j} \|e_ju_\lambda x^*- u_\lambda x^*\|_2=0$ by Lemma \ref{Lemma2.6.2}. Thus we have that $u_\lambda x^*       =u_\lambda u^* h=0$.
       The fact that the initial projection of $ u^*$   is   the range projection of $h$
        induces that $u_\lambda u^*=0$. Therefore, $u$ is a  nonzero partial isometry in $\mathcal M$  such that $H^2u \subseteq X$, $uu^*\in\mathcal D$, and $u_\lambda u^*=0$ for each $\lambda\in\Lambda$.
       This contradicts the assumption
       that the family $\{u_\lambda\}_{\lambda\in\Lambda}$ is
       maximal with respect to  (a$_1$), (b) and (c). Therefore, $X=
\oplus_{\lambda\in\Lambda} H^2u_\lambda.$ This concludes the proof
of part (iv).
    \end{proof}

\begin{lemma}\label{lemma3.5}  Let $\mathcal K$ be a closed subspace of
$L^2(\mathcal M,\tau)$ satisfying $H^\infty\mathcal K\subseteq
\mathcal K$.  Let $$X=\mathcal K\ominus[H_0 ^\infty \mathcal
K]_2  \ \text{ and } \ Y=\mathcal K\ominus [H^\infty X]_2.$$
Then the following are true.
\begin{enumerate}
  \item [(i)] $YX^*=0$, or equivalently $XY^*=0$.
  \item [(ii)] $Y=[H_0 ^\infty Y]_2$
\end{enumerate}
\end{lemma}

    \begin{proof}
      (i)   We will show that $yx^*=0$ for every $y\in Y$ and $x\in X$.

     Note that   $Y\subseteq K\subseteq L^2(\mathcal{M},\tau)$ and $X\subseteq K\subseteq L^2(\mathcal{M},\tau)$. We have that $YX^*\subseteq L^1(\mathcal{M},\tau)$. Assume $y\in Y$ and $x\in X$.
        Then by Lemma \ref{lemma3.2}, it suffices to show that $$\tau(yx^*z)=0 \qquad \text{ for every $z\in H^\infty +(H^\infty )^*$}.$$
We will proceed the proof according to the cases (1) $z\in H_0 ^{\infty}$, (2) $z\in\mathcal{D}$, and (3)
$z\in (H_0 ^\infty)^*$.

 \noindent                 Case (1):
                Let $z\in H_0 ^{\infty}$.  Then
                    \begin{equation}
                        \tau(yx^*z)=\tau(zyx^*)=0, \notag
                    \end{equation}
                since $x\in X$, $zy\in H_0 ^\infty K $, and $X\perp [H_0 ^\infty K]_2$.\\

\noindent                  Case (2):
                Let $z\in\mathcal{D}$.  Then
                    \begin{equation}
                        \tau(yx^*z)=\tau(y(z^*x)^*)=0, \notag
                    \end{equation}
                as $y\in Y$, $z^*x\in H^\infty X$, and $Y\perp H^\infty X$.\\

\noindent                  Case (3):
                Let $z\in (H_0 ^\infty)^*$.  Then\\
                    \begin{equation}
                        \tau(yx^*z)=\tau(y(z^*x)^*)=0, \notag
                    \end{equation}
                as $y\in Y$, and $z^*x\in H^\infty_0 X$.

            Therefore,  $YX^*=0$, which ends the proof of (i).

\vspace{0.2cm}

     (ii)      From part (i), we know that $YX^*=0$, whence $H_0 ^\infty YX^*=0$. Recall $Y=\mathcal K\ominus [H^\infty X]_2.$ It follows that
      $[H_0^\infty Y]_2\subseteq Y$.
      Let $Z=Y\ominus[H_0^\infty Y]_2=0$. To prove (ii),  it suffices to show that $ZZ^*=0$.
      Because
      $Z\subseteq Y,$ we have that  $Z\perp [H^\infty X]_2$, whence $Z\perp [H_0 ^\infty (Y\oplus [H^\infty X]_2)]_2$.
       This implies that  $Z\perp [H_0 ^\infty \mathcal K]_2$. Note that $X=\mathcal K\ominus [H_0 ^\infty \mathcal K]_2$. We conclude that $Z\subseteq X$.
       Note that $YX^*=0$.   Since  $Z\subseteq X$  and $Z\subseteq Y$, we have that $ZZ^*\subseteq YX^*=0$. This ends the proof of (ii).
    \end{proof}

\subsection{Proof of Theroem \ref{theorem3.1}} We are ready to prove the main result in this section. \begin{proof}
 Recall that $\mathcal K$ is a closed subspace of
$L^2(\mathcal M,\tau)$ satisfying $H^\infty\mathcal K\subseteq
\mathcal K$. Let $$X=\mathcal K\ominus[H_0 ^\infty \mathcal
K]_2  \ \text{ and } \ Y=\mathcal K\ominus [H^\infty X]_2.$$
By Lemma \ref{lemma3.4}, there exists a family $\{u_\lambda\}_{\lambda\in\Lambda}$ of partial isometries in
  $\mathcal M$ such that    $$X= \oplus_{\lambda\in\Lambda} H^2u_\lambda;$$ and \begin{equation}\LeftEqNo
     \text{$u_\lambda u_\lambda^*$ is a projection in $ \mathcal D$,   and $u_\lambda u_\mu^*=0$
for all $\lambda, \mu\in \Lambda$ with $\lambda\ne \mu$.}\tag{ ii }
\end{equation}
  By the choice of $Y$, we have
\begin{equation}\LeftEqNo
\ \qquad \qquad \mathcal K= Y\oplus X= Y\oplus (
\oplus_{\lambda\in\Lambda} H^2u_\lambda). \tag{ iv  }
\end{equation} Moreover, from Lemma \ref{lemma3.5}, we know that
\begin{equation}\LeftEqNo
 \qquad \qquad u_\lambda Y^*=0 \qquad \text{ for all } \ \lambda\in\Lambda; \tag{  i }
\end{equation}  and
\begin{equation}\LeftEqNo
 \qquad \qquad Y=[H_0 ^\infty Y]_2. \tag{ iii}
\end{equation} This ends the proof of Theorem \ref{theorem3.1}.

 \end{proof}

\section{Beurling-Blecher-Labuschagne Theorem for Semifinite Hardy Spaces, $1\leq p\leq\infty$}


\subsection{Dense subspaces}

    \begin{proposition}\label{proposition4.1}
Let $\mathcal M$ be a von Neumann algebra with a faithful, normal,
semifinite tracial weight $\tau$, and $H^\infty$  be  a semifinite
subdigonal subalgebra   of $\mathcal M$. Let $1\le p <\infty$.
Assume that $\mathcal K $ is  a closed subspace in $
L^p(\mathcal{M},\tau)$ such that $H^\infty \mathcal K \subseteq
\mathcal K$. Then the following statements are true.
            \begin{enumerate}
                \item [(i)]$\mathcal K\cap\mathcal{M} =\overline{\mathcal K\cap\mathcal{M}}^{w^*}\cap L^p(\mathcal{M},\tau)$.
                \item [(ii)]  $\mathcal K = [\mathcal K \cap \mathcal{M}]_p$.

            \end{enumerate}
     \end{proposition}

    \begin{proof}
        (i) It is easily observed that $$\mathcal K\cap\mathcal{M}\subseteq \overline{  \mathcal K\cap \mathcal{M}}^{w^*} \cap
        L^p(\mathcal{M},\tau).$$
        We will show that $$\mathcal K\cap\mathcal{M} = \overline{\mathcal K\cap \mathcal{M}}^{w^*} \cap
        L^p(\mathcal{M},\tau).$$

            Assume, to the contrary, that $\mathcal K \cap \mathcal{M} \varsubsetneqq \overline{\mathcal K\cap \mathcal{M}}^{w^*}\cap L^p(\mathcal{M},\tau)$.
              Then there exists an $x\in \overline{K_1\cap \mathcal{M}}^{w^*}\cap L^p(\mathcal{M},\tau)$ such that $x\notin   \mathcal K\cap
              \mathcal{M}$. 
            Then by the Hahn-Banach theorem, there exists $\varphi \in L^p(\mathcal{M},\tau)^\# = L^q(\mathcal{M},\tau)$
            (where $\frac{1}{p}+\frac{1}{q}=1$) such that $\varphi(x)\neq 0$ and $\varphi(y)=0 $ for
             every $y\in   \mathcal K\cap \mathcal{M}$.
            Equivalently, there exists a $\xi\in L^q(\mathcal{M},\tau)$ such that $\tau(x\xi)\neq 0$ and $\tau(y\xi)=0$ for every $y\in   \mathcal K\cap \mathcal{M}$.

         By Lemma
            \ref{lemma2.13}, there exists a net
            $\{e_\lambda\}_{\lambda\in\Lambda}$ of projections in
            $\mathcal D$ such that   $\tau(e_\lambda)<\infty$ for each $\lambda\in\Lambda$, and $\lim_\lambda \tau(e_\lambda
            x\xi)=\tau(x\xi).$ So, we can always assume that   there exists a projection $e$ in $\mathcal D$
             with $0<\tau(e)<\infty$ such that $\tau(ex\xi)\neq 0$  and $\tau(ey\xi)=0$ for every $y\in   \mathcal K$ (as
              $  \mathcal K$ is $H^\infty$-invariant and $e\in\mathcal D\subseteq H^\infty$).

             Now we claim
              that $\xi e\in L^1(\mathcal M,\tau)$, as  $||\xi e||_1\leq ||\xi ||_q
              ||e||_p<\infty$.

            Since $x\in  \overline{K_1\cap \mathcal{M}}^{w^*}\cap L^p(\mathcal{M},\tau)$, we can find a net $\{y_i\}_{i\in I}$ in
            $  \mathcal K\cap \mathcal{M}$, such that $y_i \rightarrow x $ in the weak$^*$-topology. Combining this with the fact that $\xi e\in L^1(\mathcal M,\tau)$, we
            have
                $$
                  \tau(ex\xi)=\tau(x\xi e) =\lim_i \tau (y_i \xi e)=\lim_i \tau (ey_i \xi
                  )=0, $$ which contradicts the fact that $\tau(ex\xi)\neq
                  0.$ This ends the proof of part (i).

        (ii) Suppose, to the contrary, that $[  \mathcal K\cap \mathcal{M}]_p \varsubsetneqq  \mathcal K$.
         Then there exists an $x\in   \mathcal K$ such that $x\notin [  \mathcal K \cap \mathcal{M}]_p$.
            Then, by the Hahn-Banach theorem, there exists $\varphi\in L^p(\mathcal{M},\tau)^\# = L^q(\mathcal{M},\tau)$ (where $\frac{1}{p}+\frac{1}{q}=1$), such that $\varphi(x)\neq 0$ and $\varphi(y)=0$ for every $y\in[  \mathcal K\cap \mathcal{M}]_p$.
            This occurs if and only if there exists a $\xi\in L^q(\mathcal{M},\tau)$ such that $\tau(x\xi)\neq 0$ and $\tau(y\xi)=0$ for every $y\in[  \mathcal K\cap \mathcal{M}]_p$.

 By Lemma
            \ref{lemma2.13}, there exists a net
            $\{e_\lambda\}_{\lambda\in\Lambda}$ of projections in
            $\mathcal D$ such that   $\tau(e_\lambda)<\infty$ for each $\lambda\in\Lambda$, and $\lim_\lambda \tau(e_\lambda
            x\xi)=\tau(x\xi).$  So, we may always assume that   there exists a projection $e$ in $\mathcal D$
             with $0<\tau(e)<\infty$ such that
\begin{enumerate}
\item [(a)]              $\tau(ex\xi)\neq 0$;  and
\item [(b)] $\tau(ey\xi)=0$ for every $y\in   \mathcal K\cap \mathcal{M}$
(as
              $  \mathcal K$ is $H^\infty$-invariant, and $e\in\mathcal D\subseteq
              H^\infty$).\end{enumerate}

   Since $x\in L^p(\mathcal{M},\tau)$ and $e$ is a projection in $\mathcal D$ such that $\tau(e)<\infty$, by Lemma \ref{lemma2.12},
   there exists a $h_1\in eH^\infty e$, and $h_2 \in eH^p e$ such that $h_1 ex\in \mathcal{M}$ and $h_1 h_2 = h_2 h_1=e$.  From the fact that
    $h_2\in eH^p e$,   there exists a sequence  $\{a_n\}_{n\in\mathbb N}$ in $eH^\infty e$ such that $\lim_{n\rightarrow \infty}\|a_n -h_2\|_p=0$.
Therefore
$$\begin{aligned}
\lim_{n\rightarrow \infty}|\tau(a_n h_1 ex\xi )-\tau(ex\xi) | &=
\lim_{n\rightarrow \infty}| \tau(a_n h_1 ex\xi )-\tau(h_2 h_1 ex\xi
)| \\ &\le \lim_{n\rightarrow \infty}\|a_n-h_2\|_p \|h_1ex\|
\|\xi\|_q \\ &=0.\end{aligned}
$$
On the other hand, since $a_n, h_1$ and $e$ are in $H^\infty$ and
$h_1ex \in\mathcal M$, we know that $a_nh_1 ex\in \mathcal K\cap
\mathcal M$. From assumption (b), it follows that  $\tau(a_n h_1
ex\xi ) =0$ for all $n\ge 1$.
  Therefore $\tau(ex\xi)=0$, which contradicts the assumption (a) that $\tau(x\xi e)\neq 0$.  This ends the proof of part (ii).

\end{proof}

  \begin{proposition}\label{proposition4.2}
Let $\mathcal M$ be a von Neumann algebra with a faithful, normal,
semifinite tracial weight $\tau$ and $H^\infty$  be  a semifinite
subdiagonal subalgebra   of $\mathcal M$.

Assume that  $\mathcal K \subseteq  \mathcal{M} $ is a weak$^*$-closed
subspace such that  $H^\infty \mathcal K \subseteq \mathcal K$. Then
$$  \mathcal K=\overline{[  \mathcal K \cap L^p(\mathcal{M},\tau)]_p \cap
\mathcal{M}}^{w^*}, \quad \forall \ 1\le p<\infty.$$
     \end{proposition}
     \begin{proof}
            First, we show that $$  \mathcal K\subseteq \overline{[  \mathcal K\cap L^p(\mathcal{M},\tau)]_p\cap
            \mathcal{M}}^{w*}.$$
        Let $x$ be an element in $  \mathcal K\subseteq \mathcal{M}$.
  By Lemma
            \ref{lemma2.13}, there exists a net
            $\{e_\lambda\}_{\lambda\in\Lambda}$ of projections in
            $\mathcal D$ such that  such that $e_\lambda \rightarrow I$ in the weak* topology  and $\tau(e_\lambda)<\infty$ for each $\lambda\in\Lambda$.
           By Lemma \ref{Lemma2.1}, $e_\lambda x \rightarrow x$ in the weak*
           topology. To show that $x\in \overline{[  \mathcal K\cap L^p(\mathcal{M},\tau)]_p\cap
            \mathcal{M}}^{w*},$ it suffices to show that $e_\lambda x\in \overline{[  \mathcal K\cap L^p(\mathcal{M},\tau)]_p\cap
            \mathcal{M}}^{w*} $ for each $\lambda\in\Lambda.$

Since $  \mathcal K\subseteq \mathcal{M}$ is left
$H^\infty$-invariant and $x\in \mathcal K$, we have $e_\lambda x\in
\mathcal K$. Moreover, $\|e_\lambda x||_p\leq \|e_\lambda||_p
||x||<\infty$, so $e_\lambda x\in L^p(\mathcal{M},\tau)$. It follows
that $e_\lambda x\in   \mathcal K\cap L^p(\mathcal{M},\tau)  $ for
each $\lambda\in\Lambda.$ As $e_\lambda x \rightarrow x$ in the
weak*                         topology,  $x\in \overline{[  \mathcal
K\cap L^p(\mathcal{M},\tau)]_p\cap
            \mathcal{M}}^{w*} $. And we
            obtain  $  \mathcal K\subseteq \overline{[  \mathcal K\cap L^p(\mathcal{M},\tau)]_p\cap
            \mathcal{M}}^{w*}.$

         Next, we will show that
$$  \overline{[  \mathcal K\cap L^p(\mathcal{M},\tau)]_p\cap
            \mathcal{M}}^{w*}\subseteq \mathcal K .$$
Since $\mathcal K$ is weak$^*$-closed, it suffices to show that
$$   [  \mathcal K\cap L^p(\mathcal{M},\tau)]_p\cap
            \mathcal{M} \subseteq \mathcal K .$$
 Assume, to the contrary, that $x$ is an element in $ {[  \mathcal K\cap L^p(\mathcal{M},\tau)]_p\cap
\mathcal{M}} $, but $x\notin   \mathcal K$.
            Thus, by the Hahn-Banach theorem, there exists a weak* continuous linear functional $\varphi$ on $\mathcal{M}$ such that $\varphi(x)\neq 0$ and $\varphi(y)=0$ for every $y\in   \mathcal K$.
            Or, there exists a $\xi \in L^1(\mathcal{M},\tau)$ such that
\begin{enumerate}
\item [(a)]
            $\tau(x\xi)\neq 0$; and
            \item [(b)] $\tau(y\xi)=0$ for every $y\in   \mathcal K$. \end{enumerate}
By Lemma
            \ref{lemma2.13}, there exists a net
            $\{e_\lambda\}_{\lambda\in\Lambda}$ of projections in
            $\mathcal D$ such that   $\tau(e_\lambda)<\infty$ for each $\lambda\in\Lambda$ and $\lim_\lambda \tau(e_\lambda
            x\xi)=\tau(x\xi).$  So we may always assume that   there exists a projection $e$ in $\mathcal D$
             with $0<\tau(e)<\infty$ such that
\begin{enumerate}
\item [(a$_1$)]              $\tau(ex\xi)\neq 0$;  and
\item [(b$_1$)] $\tau(ey\xi)=0$ for every $y\in   \mathcal K $
(as
              $  \mathcal K$ is $H^\infty$-invariant and $e\in\mathcal D\subseteq
              H^\infty$).\end{enumerate}
We claim there exists a $z=ze\in \mathcal{M}e$ such that
\begin{enumerate}
\item [(a$_2$)]              $\tau( xz)\neq 0$;  and
\item [(b$_2$)] $\tau( yz)=0$ for every $y\in   \mathcal K $. \end{enumerate}
Observe that  $\xi $ is in $L^1(\mathcal M,\tau)$, and $e$ is a
projection in $\mathcal D$ such that $\tau(e)<\infty$. From Lemma
\ref{lemma2.12},   there exist  $h_3\in eH^\infty e$  and $h_4\in
eH^1 e$ such that $\xi eh_3\in e\mathcal{M}e$ and $h_3h_4=e$.
            Thus there exists a sequence  $\{k_n\}_{n\in\mathbb N}$ of elements in $eH^\infty e$ such that $\lim_{n\rightarrow \infty}\|k_n -
             h_4\|_1=0$. It follows that
             $$
\lim_{n\rightarrow \infty}|\tau(ex\xi)-\tau(x\xi e h_3
k_n)|=\lim_{n\rightarrow \infty}|\tau( x\xi e h_3h_4)-\tau(x\xi e
h_3 k_n)| \le \lim_{n\rightarrow \infty} \| x \| \|\xi e h_3
\|\|h_4-k_n\|_1=0.
             $$
Combining this with (a$_1$), we know that  there exists an $N\in\mathbb
N$ such that $\tau(x\xi e h_3 k_N)\ne 0$. Let $z=(\xi e h_3) k_N $
be in $\mathcal M $. Then $z=ze\in\mathcal Me$  satisfies
\begin{enumerate}
\item [(a$_2$)]              $\tau( xz) =\tau(x\xi e h_3 k_N)\neq 0$;  and
\item [(b$_2$)] $\tau( yz)=\tau(y\xi e h_3 k_N)= \tau((e h_3 k_N)y\xi )  =0$ for every $y\in   \mathcal K $. \end{enumerate}

Note that  $x\in  [  \mathcal K \cap L^p(\mathcal{M},\tau)]_p \cap
\mathcal{M} $.
            There exists a sequence $\{x_n\}_{n\in\mathcal N}$ in $   \mathcal K\cap L^p(\mathcal{M},\tau) $
             such that $\lim_{n\rightarrow \infty}\|x_n-x\|_p=0.$
Thus we have
\begin{align}
|\tau(xz-x_nz) |= |\tau((x-x_n)ze) |\le \|x_n-x\|_p \|z\|
\|e\|_q\rightarrow 0, \label{equation4.1}\end{align} where $q$
satisfies $1/p+1/q=1$. On the other hand, since
$\{x_n\}_{n\in\mathcal N}$ is in $ \mathcal K\cap
L^p(\mathcal{M},\tau) $, by (b$_2$) we have
$$
\tau( x_nz)=0, \qquad \forall \ n\in\mathbb N.
$$Combining with inequality (\ref{equation4.1}), we have $$\tau(xz)=0.$$ This contradicts the assumption in (a$_2$) that     $\tau( xz)\neq 0$.
Therefore,$$  \overline{[  \mathcal K\cap
L^p(\mathcal{M},\tau)]_p\cap
            \mathcal{M}}^{w*}\subseteq \mathcal K .$$

Hence $$  \mathcal K =\overline{[  \mathcal K\cap
L^p(\mathcal{M},\tau)]_p\cap
            \mathcal{M}}^{w*} .$$
\end{proof}

\begin{lemma}\label{lemma4.3.1}If $u$ is a partial isometry in
$\mathcal M$ such that $uu^*\in\mathcal D$, then
\begin{enumerate}
\item [(i)] $\displaystyle [(H^\infty u)\cap L^p(\mathcal M,\tau)]_p=H^pu    \text{
for all } 1\le p<\infty,$  and \item [(ii)] $ \displaystyle H^\infty
u= \overline{H^pu\cap \mathcal M}^{w^*}    \text{ for all } 1\le
p<\infty. $ \end{enumerate}
\end{lemma}
\begin{proof}
(i) can be verified directly. (ii) follows from Proposition
\ref{proposition4.2} and (i).
\end{proof}

 \begin{proposition}\label{prop4.4}
Let $\mathcal M$ be a von Neumann algebra with a faithful, normal,
semifinite tracial weight $\tau$ and let $H^\infty$  be  a semifinite
subdigonal subalgebra   of $\mathcal M$.  Assume that  $  S
\subseteq \mathcal{M} $ is a subspace such that  $H^\infty  S
\subseteq  S$.   Then
          $$[S\cap L^p(\mathcal{M},\tau)]_p = [\overline{S}^{w^*} \cap L^p(\mathcal{M},\tau)]_p, \quad \forall \ 1\le p<\infty.$$

     \end{proposition}
\begin{proof} It suffices to show that
$$\overline{S}^{w^*} \cap L^p(\mathcal{M},\tau) \subseteq [S\cap
L^p(\mathcal{M},\tau)]_p.$$ Let $x\in \overline{S}^{w^*} \cap
L^p(\mathcal{M},\tau)$.  By Lemma
            \ref{lemma2.13}, there exists a net
            $\{e_\lambda\}_{\lambda\in\Lambda}$ of projections in
            $\mathcal D$ such that $e_\lambda \rightarrow I$ in the weak* topology  and $\tau(e_\lambda)<\infty$ for each $\lambda\in\Lambda$.
           By Lemma \ref{Lemma2.6.2}, $\lim_\lambda \|e_\lambda
           x-x\|_p=0.$ To show that $x\in [S\cap
L^p(\mathcal{M},\tau)]_p$, it is enough to show that $e_\lambda x\in
[S\cap L^p(\mathcal{M},\tau)]_p$ for each $\lambda\in\Lambda$.

By Proposition \ref{proposition4.1}, we have
$$
[S\cap L^p(\mathcal{M},\tau)]_p \cap \mathcal M= \overline{[S\cap
L^p(\mathcal{M},\tau)]_p \cap \mathcal M}^{w^*}\cap L^p(\mathcal
M,\tau).
$$
Since $x\in \overline{S}^{w^*} \cap L^p(\mathcal{M},\tau)$, there
exists a net $\{x_j\}_{j\in J}$ in $S$ such that $x_j\rightarrow x$
in weak$^*$ topology. By Lemma \ref{Lemma2.1}, $e_\lambda
x_j\rightarrow e_\lambda x$ in weak$^*$ topology for each $\lambda$.
Note that $\|e_\lambda x_j\|_p\le  \|e_\lambda \|_p \|x_j\|$ and
$H^\infty S\subseteq S$. We know that $e_\lambda x_j\in S\cap
L^p(\mathcal{M},\tau)$. So $e_\lambda x$ is in $\overline{[S\cap
L^p(\mathcal{M},\tau)]_p \cap \mathcal M}^{w^*}$. It is trivial to
see that $e_\lambda x\in L^p(\mathcal{M},\tau)$. Hence,
$$
e_\lambda x\in \overline{[S\cap L^p(\mathcal{M},\tau)]_p \cap
\mathcal M}^{w^*}\cap L^p(\mathcal M,\tau)= [S\cap
L^p(\mathcal{M},\tau)]_p.
$$ So
$$x\in
[S\cap L^p(\mathcal{M},\tau)]_p.$$ Thus $$\overline{S}^{w^*} \cap
L^p(\mathcal{M},\tau) \subseteq [S\cap L^p(\mathcal{M},\tau)]_p.$$
Hence
  $$[S\cap L^p(\mathcal{M},\tau)]_p = [\overline{S}^{w^*} \cap L^p(\mathcal{M},\tau)]_p, \quad \forall \ 1\le p<\infty.$$

    \end{proof}

        \begin{theorem}\label{theorem4.5}
Let $\mathcal M$ be a von Neumann algebra with a faithful, normal,
semifinite tracial weight $\tau$, and $H^\infty$  be  a semifinite
subdigonal subalgebra   of $\mathcal M$.  Let $\mathcal
D=H^\infty\cap (H^\infty)^*$. Assume that $\mathcal K \subseteq
\mathcal{M} $ is weak$^*$-closed subspace such that $H^\infty
\mathcal K \subseteq \mathcal K$.

             Then there exist   a weak* closed subspace $Y$ of $\mathcal M$ and a family $\{u_\lambda\}_{\lambda \in\Lambda}$ of partial isometries in
             $ \mathcal{M}$  such that:
            \begin{enumerate}
              \item  [(i)] $u_\lambda Y^*=0$ for all ${\lambda\in\Lambda}$.
\item  [(ii)] $u_\lambda u_\lambda^*\in \mathcal D$ and $u_\lambda u_\mu^*=0$
for all $\lambda, \mu\in \Lambda$ with $\lambda\ne \mu$.
\item  [(iii)] $Y=\overline{H^\infty_0Y}^{w^*}$.
                \item [(iv)] $\mathcal K=Y \oplus^{row} (\oplus^{row}_{\lambda\in\Lambda } H^\infty u_\lambda)$
            \end{enumerate} Here $\oplus^{row}$ is the row sum of
            subspaces
            defined in Definition \ref{definition2.15}.
        \end{theorem}

    \begin{proof}
Let $\mathcal K_1= [\mathcal K \cap L^2(\mathcal{M},\tau)]_2$. Then
$\mathcal K_1$ is a closed subspace of $L^2(\mathcal M,\tau)$ such
that
  $H^\infty \mathcal K_1\subseteq \mathcal K_1$.  By Theorem \ref{theorem3.1}, there exist a closed subspace $Y_1$ of $L^2(\mathcal
M,\tau)$ and a family $\{u_\lambda\}_{\lambda\in\Lambda}$ of partial
isometries in $\mathcal M$, satisfying
\begin{enumerate}
\item  [(a)] $u_\lambda Y_1^*=0$ for all ${\lambda\in\Lambda}$.
\item  [(b)] $u_\lambda u_\lambda^*\in \mathcal D$ and $u_\lambda u_\mu^*=0$
for all $\lambda, \mu\in \Lambda$ with $\lambda\ne \mu$.
\item  [(c)] $Y_1=[H^\infty_0Y_1]_2$, where $H^\infty_0=H^\infty\cap
ker(\Phi)$.
\item [(d)] $\mathcal K_1=Y_1\oplus \left ( \oplus_{\lambda\in\Lambda} H^2u_\lambda \right
)$
\end{enumerate}
Let
$$
Y=\overline{Y_1\cap \mathcal{M}}^{w^*}.
$$

(i) We show that (i) is satisfied. In fact, from (a) and Lemma
\ref{Lemma2.1}, we have
\begin{equation}
  \text{$u_\lambda Y^*=0$ for all ${\lambda\in\Lambda}$.}  \label{equa4.4.2}
\end{equation}

(ii)   follows directly from
 (b), i.e.
 \begin{equation}
  \text{ $u_\lambda u_\lambda^*\in \mathcal D$ and $u_\lambda u_\mu^*=0$
for all $\lambda, \mu\in \Lambda$ with $\lambda\ne \mu$.}  \label{equa4.4.3}
\end{equation}

(iii) We claim that
\begin{equation}
  Y=\overline{H^\infty_0Y}^{w^*}. \notag
\end{equation}
In fact, we need only to show that $Y\subseteq\overline{H_0^\infty
Y}^{w^*}$. By Proposition \ref{proposition4.1} and the definition of
$Y$, we have
\begin{align}
Y_1=[Y_1\cap\mathcal M]_2= [\overline {Y_1\cap\mathcal M}^{w*}\cap
L^2(\mathcal M,\tau)]_2=[Y\cap L^2(\mathcal M,\tau)]_2. \notag
\end{align} So
\begin{align}
H^\infty_0 Y_1= H^\infty_0 [Y\cap L^2(\mathcal M,\tau)]_2\subseteq
 [(H^\infty_0Y)\cap L^2(\mathcal M,\tau)]_2\subseteq
 [\overline{H^\infty_0Y}^{w^*}\cap L^2(\mathcal M,\tau)]_2.\notag
\end{align}
Thus, from (c), we have
\begin{align}
Y_1=[H^\infty_0 Y_1]_2\subseteq [\overline{H^\infty_0Y}^{w^*}\cap
L^2(\mathcal M,\tau)]_2.\label {equation4.2}
\end{align}
Now, we are able to conclude that
            \begin{align}
                Y &=\overline{Y_1\cap \mathcal{M}}^{w^*}\tag{by definition of $Y$}\\
                &\subseteq \overline{[\overline{H^\infty_0Y}^{w^*}\cap
L^2(\mathcal M,\tau)]_2\cap \mathcal{M}}^{w^*} \tag{by (\ref{equation4.2})}\\
                  & = \overline{H_0^\infty Y}^{w^*}. \tag {by Proposition \ref{proposition4.2}}
            \end{align}
           Thus
\begin{equation}
  Y=\overline{H^\infty_0Y}^{w^*}.   \label{equa4.4.4}
\end{equation}

(iv)   We show that
                $$\overline{span\{Y_2, H^\infty u_\lambda:\lambda\in\Lambda\}}^{w^*}=\mathcal K .$$
                By Proposition \ref{proposition4.2}, it suffices to
                show that
 $$\overline{span\{Y , H^\infty u_\lambda  : \lambda\in\Lambda\}}^{w^*}=\overline{[\mathcal K \cap L^2(\mathcal{M},\tau)]_2\cap
 \mathcal{M}}^{w^*}=\mathcal K.$$
First, we have that $span\{Y , H^\infty
u_\lambda:\lambda\in\Lambda\} \subseteq \overline{[\mathcal K \cap
L^2(\mathcal{M},\tau)]_2\cap
 \mathcal{M}}^{w^*}$.  In fact, $Y=\overline{Y_1\cap \mathcal{M}}^{w^*}$  and $Y_1\subseteq [\mathcal K \cap L^2(\mathcal{M},\tau)]_2$,
  so $Y \subseteq  \overline{[\mathcal K \cap L^2(\mathcal{M},\tau)]_2\cap
 \mathcal{M}}^{w^*}$. Moreover, for each $\lambda\in\Lambda$, by Lemma \ref{lemma4.3.1}, we have $H^\infty u_\lambda =\overline{H^2u_\lambda\cap
  \mathcal M}^{w^*}\subseteq \overline{[\mathcal K \cap L^2(\mathcal{M},\tau)]_2\cap
 \mathcal{M}}^{w^*}$.  So $$span\{Y , H^\infty u_\lambda : \lambda\in\Lambda\} \subseteq
\overline{[\mathcal K \cap L^2(\mathcal{M},\tau)]_2\cap
 \mathcal{M}}^{w^*}.$$
 Thus  \begin{align}
 \overline{span\{Y , H^\infty u_\lambda : \lambda\in\Lambda\}}^{w^*}\subseteq\overline{[\mathcal K \cap L^2(\mathcal{M},\tau)]_2\cap
 \mathcal{M}}^{w^*}=\mathcal K. \label{equa4.3}\end{align}
        Next, define $X=\overline{span\{Y, H^\infty u_\lambda : \lambda\in\Lambda\}}^{w^*}$. We want to show that $$\overline{[\mathcal K \cap L^2(\mathcal{M},\tau)]_2\cap
 \mathcal{M}}^{w^*} \subseteq X.$$
            Notice $X$ is weak*-closed and $H^\infty X\subseteq X$. By Proposition \ref{proposition4.2},  $X=\overline{[X \cap L^2(\mathcal{M},\tau)]_2\cap \mathcal{M}}^{w^*}$.
            Therefore we need only to show that $[\mathcal K \cap L^2(\mathcal{M},\tau)]_2\subseteq [X\cap L^2(\mathcal{M},\tau)]_2$. Or,
            equivalently, we may show
$Y_1$ and $\{H^2u_\lambda\}_{\lambda\in\Lambda}$ are in $[X\cap
L^2(\mathcal{M},\tau)]_2$. By Proposition \ref{proposition4.1}, we
have
\begin{align}
Y_1=[Y_1\cap\mathcal M]_2= [\overline {Y_1\cap\mathcal M}^{w*}\cap
L^2(\mathcal M,\tau)]_2=[Y\cap L^2(\mathcal M,\tau)]_2. \notag
\end{align} Thus \begin{align}Y_1  \subseteq [X\cap L^2(\mathcal{M},\tau)]_2.\label{equa4.2.1}\end{align} By Lemma \ref{lemma4.3.1},
\begin{align}H^2u_\lambda= [H^\infty u_\lambda\cap L^2(\mathcal M,\tau)]_2 \subseteq
[X\cap L^2(\mathcal{M},\tau)]_2\quad \text{for each
$\lambda\in\Lambda$.}\label{equa4.2.2}\end{align}  Hence, from
(\ref{equa4.2.1}) and (\ref{equa4.2.2}),  we get $[\mathcal K \cap
L^2(\mathcal{M},\tau)]_2\subseteq [X\cap L^2(\mathcal{M},\tau)]_2$
and
\begin{align}
\mathcal K= \overline{[\mathcal K \cap L^2(\mathcal{M},\tau)]_2\cap
 \mathcal{M}}^{w^*} \subseteq
\overline{span\{Y, H^\infty
u_\lambda\lambda\in\Lambda\}}^{w^*}.\label{equa4.4}
\end{align}
Now, combining   (\ref{equa4.3}) and    (\ref{equa4.4}), we have
\begin{equation}
  \mathcal K  =
\overline{span\{Y, H^\infty u_\lambda\lambda\in\Lambda\}}^{w^*} = Y
\oplus^{row} (\oplus^{row}_{\lambda\in\Lambda } H^\infty u_\lambda)
, \label{equa4.4.1}
\end{equation} by the definition of row sum of subspaces.

By (\ref{equa4.4.1}), (\ref{equa4.4.2}), (\ref{equa4.4.3}) and (\ref{equa4.4.4}), we know that $Y$ and
$\{u_\lambda\}_{\lambda \in\Lambda}$ have the desired properties. This
ends the proof of the theorem.
               \end{proof}

    Next, we use our result for $p=\infty$ and the density theorem to prove the case when $1\leq p<\infty$.

        \begin{theorem} \label{theorem4.6} Let $1\le p<\infty$.
Let $\mathcal M$ be a von Neumann algebra with a faithful, normal,
semifinite tracial weight $\tau$, and $H^\infty$  be  a semifinite
subdigonal subalgebra   of $\mathcal M$.  Let $\mathcal
D=H^\infty\cap (H^\infty)^*$. Assume that $\mathcal K  $ is a closed
subspace of $L^p(\mathcal M,\tau)$ such that $H^\infty \mathcal K
\subseteq \mathcal K$.

             Then there exist   a   closed subspace $Y$ of $L^p(\mathcal M,\tau)$ and a family $\{u_\lambda\}_{\lambda \in\Lambda}$ of partial isometries in
             $ \mathcal{M}$  such that:
            \begin{enumerate}
              \item  [(i)] $u_\lambda Y^*=0$ for all ${\lambda\in\Lambda}$.
\item  [(ii)] $u_\lambda u_\lambda^*\in \mathcal D$ and $u_\lambda u_\mu^*=0$
for all $\lambda, \mu\in \Lambda$ with $\lambda\ne \mu$.
\item  [(iii)] $Y= [H^\infty_0Y]_p$.
                \item [(iv)] $\mathcal K=Y \oplus^{row} (\oplus^{row}_{\lambda\in\Lambda } H^p u_\lambda)$
            \end{enumerate} Here $\oplus^{row}$ is the row sum of
            subspaces
            defined in Definition \ref{definition2.15}.
        \end{theorem}

    \begin{proof}
Let $\mathcal K_1= \overline{\mathcal K \cap  \mathcal{M} }^{w^*}$.
Then $\mathcal K_1$ is a  weak$^*$-closed subspace of $\mathcal{M}$
such that
  $H^\infty \mathcal K_1\subseteq \mathcal K_1$.  By Theorem \ref{theorem4.5}, there exist a weak$^*$-closed subspace $Y_1$ of $ \mathcal
M $ and a family $\{u_\lambda\}_{\lambda\in\Lambda}$ of partial
isometries in $\mathcal M$, satisfying
\begin{enumerate}
              \item  [(a)] $u_\lambda Y_1^*=0$ for all ${\lambda\in\Lambda}$.
\item  [(b)] $u_\lambda u_\lambda^*\in \mathcal D$ and $u_\lambda u_\mu^*=0$
for all $\lambda, \mu\in \Lambda$ with $\lambda\ne \mu$.
\item  [(c)] $Y_1=\overline{H^\infty_0Y_1}^{w^*}$.
                \item [(d)] $\mathcal K_1=Y_1 \oplus^{row} (\oplus^{row}_{\lambda\in\Lambda } H^\infty u_\lambda)$
\end{enumerate}
Let
$$
Y=[{Y_1\cap L^p(\mathcal{M},\tau})]_p.
$$

(i) From (a), the definition of $Y$ and Lemma \ref{Lemma2.5}, we can
conclude that
\begin{align}
\text{$u_\lambda Y^*=0$ for all
${\lambda\in\Lambda}$.\label{equa4.11}}
\end{align}

(ii)   follows directly from
 (b), i.e.
 \begin{equation}
  \text{ $u_\lambda u_\lambda^*\in \mathcal D$ and $u_\lambda u_\mu^*=0$
for all $\lambda, \mu\in \Lambda$ with $\lambda\ne \mu$.}
\label{equa4.12}
\end{equation}

(iii) We want to show that $Y= [H^\infty_0Y]_p$. In fact, we have
 \begin{align}
                Y &=[Y_1  \cap L^p(\mathcal{M},\tau)]_p\tag{by definition of $Y$}\\
                &=[\overline{H_0^\infty Y_1}^{w^*}\cap L^p(\mathcal{M},\tau)]_p\tag{by (c)}\\
                &=[(H_0^\infty Y_1)\cap L^p(\mathcal{M},\tau)]_p\tag{by Proposition \ref{prop4.4}}\\
                &=[\left (H_0^\infty \overline{[Y_1  \cap L^p(\mathcal{M},\tau)]_p \cap \mathcal M}^{w^*}\right)\cap L^p(\mathcal{M},\tau)]_p
                \tag{by Proposition \ref{proposition4.2}}\\
&\subseteq [ \overline{H_0^\infty \left ( [Y_1  \cap
L^p(\mathcal{M},\tau)]_p \cap \mathcal M\right)}^{w^*}\cap
L^p(\mathcal{M},\tau)]_p
                \tag{by Lemma  \ref{Lemma2.1}}\\
                &= [ \left ( H_0^\infty \left ( [Y_1  \cap
L^p(\mathcal{M},\tau)]_p \cap \mathcal M\right) \right )\cap
L^p(\mathcal{M},\tau)]_p
                \tag{by Proposition \ref{prop4.4}}\\
                 &= [ \left ( H_0^\infty \left ( Y\cap \mathcal M\right) \right )\cap
L^p(\mathcal{M},\tau)]_p
                \tag{by definition of $Y$}\\
                &\subseteq [H_0^\infty Y]_p\subseteq Y,\label{equa4.12.1}
            \end{align}

(iv) There is only left to show that $$\mathcal K=Y \oplus^{row}
(\oplus^{row}_{\lambda\in\Lambda } H^p u_\lambda).$$  By   the
definition of $Y$, we have
\begin{align}
Y=[Y_1\cap L^p(\mathcal M,\tau)]_p .\label{equa4.13}
            \end{align} And from Lemma \ref{lemma4.3.1}, we have
\begin{align}
H^pu_\lambda =[H^\infty u_\lambda \cap L^p(\mathcal M,\tau)]_p , \ \
\forall \ \lambda\in\Lambda.\label{equa4.14}
            \end{align} Now, we have
\begin{align}
\mathcal K&= [\mathcal K_1\cap L^p(\mathcal M,\tau)]_p \tag {by
Proposition \ref{proposition4.1}}\\
 &=[ \overline{span\{Y_1,
H^\infty u_\lambda : \lambda \in\Lambda \}}^{w^*}\cap L^p(\mathcal
M,\tau) ]_p \tag{by the definition of row sum of subspaces} \\
&=[
 {span\{Y_1, H^\infty u_\lambda : \lambda \in\Lambda \}} \cap
L^p(\mathcal M,\tau) ]_p\tag{by  Proposition \ref{prop4.4}}\\
&=[
 {span\{Y_1\cap
L^p(\mathcal M,\tau), H^\infty u_\lambda \cap L^p(\mathcal M,\tau):
\lambda \in\Lambda \}} ]_p\tag{by   (a) and (b)}\\
&  =[
 {span\{Y  , H^p u_\lambda  :
\lambda \in\Lambda \}} ]_p \tag  {by (\ref{equa4.13}) and
(\ref{equa4.14}) } \\& =Y \oplus^{row}
(\oplus^{row}_{\lambda\in\Lambda } H^p u_\lambda), \label{equa4.16}
\end{align}  where the last equation follows from the definition of the row sum of
  subspaces.

As a summary, from (\ref{equa4.11}), (\ref{equa4.12}),
(\ref{equa4.12.1}), and (\ref{equa4.16}), $Y$ and
$\{u_\lambda\}_{\lambda \in\Lambda}$ have the desired properties. This
ends the proof of the theorem.
    \end{proof}

\section{Beurling-Blecher-Labuschagne Theorem for Semifinite Hardy
Spaces,  $0<p<1$}

%

 \subsection{Dense subspaces}

        \begin{proposition}\label{proposition5.1}
            Suppose $0<p<1$.
 Let $\mathcal M$ be a von Neumann algebra
with a faithful, normal, semifinite tracial weight $\tau$, and
$H^\infty$  be  a semifinite subdigonal subalgebra   of $\mathcal
M$.   Assume that $\mathcal K $ is  a closed subspace in $
L^p(\mathcal{M},\tau)$ such that $H^\infty \mathcal K \subseteq
\mathcal K$. Then the following statements are true.
            \begin{enumerate}
                \item [(i)] $\mathcal K\cap L^2(\mathcal{M},\tau)=[\mathcal K\cap L^2(\mathcal{M},\tau)]_2\cap L^p(\mathcal{M},\tau)$.
                \item [(ii)] $\mathcal K=[\mathcal K\cap L^2(\mathcal{M},\tau)]_p$.

            \end{enumerate}
     \end{proposition}
\begin{proof}
(i) We need only to show that $$ [\mathcal K\cap
L^2(\mathcal{M},\tau)]_2\cap L^p(\mathcal{M},\tau) \subseteq
\mathcal K\cap L^2(\mathcal{M},\tau).$$ Let $x\in [K_1\cap
L^2(\mathcal{M},\tau)]_2\cap L^p(\mathcal{M},\tau)$. We will show
that $x\in\mathcal K$.   By Lemma
            \ref{lemma2.13}, there exists a net
            $\{e_\lambda\}_{\lambda\in\Lambda}$ of projections in
            $\mathcal D$ such that  such that   $\tau(e_\lambda)<\infty$ for each $\lambda\in\Lambda$ and  $\lim_\lambda\|e_\lambda x - x\|_p=0.$   To show that $x\in \mathcal K,$ it is enough to prove that $e_\lambda x\in \mathcal K $ for each $\lambda\in\Lambda.$

As $x\in[\mathcal K\cap L^2(\mathcal{M},\tau)]_2$,  there exists a
sequence $\{x_n\}_{n\in\mathbb N}$ in $K_1\cap
L^2(\mathcal{M},\tau)$ such that $\lim_{n\rightarrow \infty
}\|x_n\rightarrow x\|_2=0.$ Thus, for each $\lambda\in\Lambda$ and
  some positive number $q$ with
                $\frac{1}{2}+\frac{1}{q}=\frac{1}{p}$,
               $$\lim_{n\rightarrow \infty}\|e_\lambda x_n - e_\lambda x\|_p=\lim_{n\rightarrow \infty}\|e_\lambda (x_n-x)\|_p\leq \lim_{n\rightarrow \infty}\|x_n-x\|_2
               \|e_\lambda\|_q =0.$$ Here, we used the fact that  $\tau(e_\lambda)<\infty$ and
               $\|e_\lambda\|_q<\infty$. Since $H^\infty \mathcal K \subseteq
\mathcal K$ and $e_\lambda\in \mathcal D$, we know that $e_\lambda
x_n\in \mathcal K$. This implies that $e_\lambda x\in \mathcal K$
for each $\lambda\in\Lambda$. Thus $x\in\mathcal K$, whence $$
[\mathcal K\cap L^2(\mathcal{M},\tau)]_2\cap L^p(\mathcal{M},\tau)
\subseteq \mathcal K\cap L^2(\mathcal{M},\tau).$$

(ii)  We need only to show that $$\mathcal K\subseteq [\mathcal K
\cap L^2(\mathcal{M},\tau)]_p  .$$ Suppose that $x\in \mathcal K
\subseteq L^p(\mathcal{M},\tau)$. We will show $x\in   [\mathcal K
\cap L^2(\mathcal{M},\tau)]_p $.  By Lemma
            \ref{lemma2.13},  we can
find a net $\{e_\lambda\}_{\lambda\in\Lambda}$ of projections in $
\mathcal{D}$ such that $\lim_\lambda \|e_\lambda x- x\|_2=0$ and
$\tau (e_\lambda)<\infty$ for each $\lambda \in \Lambda$. To show
that $x\in [\mathcal K \cap L^2(\mathcal{M},\tau)]_p  $, it suffices
to prove that $e_\lambda x\in [\mathcal K \cap
L^2(\mathcal{M},\tau)]_p $ for each $\lambda \in \Lambda$.

Note that $x\in L^p(\mathcal M,\tau)$ and $\tau(e_\lambda)<\infty$.
By Lemma \ref{lemma2.12},  there exist $h_1\in e_\lambda H^\infty
e_\lambda  $ and $h_2\in e_\lambda H^pe_\lambda $ such that (a)
$h_1h_2=h_2h_1=e_\lambda$ and (b) $h_1e_\lambda x\in  \mathcal{M} $.
               Since $h_2\in e_\lambda H^pe_\lambda $,  there exists a sequence $\{k_n\}_{n\in\mathbb N}$
                in $e_\lambda H^\infty
e_\lambda  $ such that $\lim_{n\rightarrow \infty} \|k_n-
h_2\|_p=0.$ Thus
\begin{align}
\lim_{n\rightarrow \infty} \|k_nh_1e_\lambda x- e_\lambda
x\|_p=\lim_{n\rightarrow \infty} \|(k_n-h_2)h_1e_\lambda x \|_p\le
\lim_{n\rightarrow \infty} \|(k_n-h_2)\|_p\|h_1e_\lambda
x\|=0.\label{equa5.1}
\end{align}
   It is not hard to check that  $k_nh_1e_\lambda x\in  \mathcal K$. Moreover, since each $k_n\in e_\lambda H^\infty
e_\lambda$,  we have
   $$
\|k_nh_1e_\lambda x\|_p= \|e_\lambda k_nh_1e_\lambda x\|_2 \le
\|e_\lambda\|_2 \|k_n\|\|h_1e_\lambda x\|<\infty.
   $$    Therefore,  $k_nh_1e_\lambda x$ is also in $  L^2(\mathcal{M},\tau)$. It follows that   $k_nh_1e_\lambda x\in \mathcal K\cap
   L^p(\mathcal{M},\tau)$.   Combining with (\ref{equa5.1}), we know that $e_\lambda x\in   [\mathcal K \cap
L^2(\mathcal{M},\tau)]_p $ for each $\lambda \in \Lambda$, whence
$x\in [\mathcal K \cap L^2(\mathcal{M},\tau)]_p $. Thus
$$\mathcal K\subseteq [\mathcal K
\cap L^2(\mathcal{M},\tau)]_p  .$$

This ends the proof of the proposition.
\end{proof}

   \begin{proposition}\label{proposition5.2}
            Suppose $0<p<1$.
 Let $\mathcal M$ be a von Neumann algebra
with a faithful, normal, semifinite tracial weight $\tau$, and
$H^\infty$  be  a semifinite subdigonal subalgebra   of $\mathcal
M$.   Assume that $S $ is  a   subspace in $ L^p(\mathcal{M},\tau)$
such that $H^\infty S \subseteq S$. Then $$[S\cap
L^p(\mathcal{M},\tau)]_p=[[S]_2\cap L^p(\mathcal{M},\tau)]_p.$$

     \end{proposition}

        \begin{proof}

          We need only  to show that $$[[S]_2\cap L^p(\mathcal{M},\tau)]_p\subseteq[S\cap
          L^p(\mathcal{M},\tau)]_p.$$ Or, equivalently,
$$ [S]_2\cap L^p(\mathcal{M},\tau) \subseteq[S\cap
          L^p(\mathcal{M},\tau)]_p.$$
Let $x\in [S]_2\cap L^p(\mathcal{M},\tau) .$
                By Lemma
            \ref{lemma2.13},   we can
find a net $\{e_\lambda\}_{\lambda\in\Lambda}$ of projections in $
\mathcal{D}$ such that $\lim_\lambda \|e_\lambda x- x\|_p=0$ and
$\tau (e_\lambda)<\infty$ for each $\lambda \in \Lambda$. To show
that $x\in [S\cap L^p(\mathcal{M},\tau)]_p  $, it suffices to prove
that $e_\lambda x\in [S\cap L^p(\mathcal{M},\tau)]_p $ for each
$\lambda \in \Lambda$.

                Note that  $x\in [S]_2\cap L^p(\mathcal{M},\tau) .$  Then there exists a sequence $\{x_n\}_{n\in\mathbb N}$
                 in $S$ such that $\lim_{n\rightarrow \infty}\|x_n-
                 x\|_2=0.$ Therefore,
\begin{align}\|e_\lambda x_n-e_\lambda  x\|_p=\|e_\lambda( x_n - x)\|_p\leq
\|e_\lambda\|_q \|  x_n -   x\|_2 \rightarrow 0, \qquad \text{ as
$n\rightarrow \infty,$}\label{equa5.2.1}\end{align} where $q$ is a
positive number such that $\frac{1}{2}+\frac{1}{q}=\frac{1}{p}$.
Since $H^\infty S\subseteq S$ and $e_\lambda\in\mathcal D$, we know
that $e_\lambda x_n\in S$. Moreover, $\|e_\lambda x_n\|_p\le
\|e_\lambda\|_q\|x_n\|_2<\infty,$ which implies $e_\lambda x_n\in
L^p(\mathcal M,\tau)$. This induces that $e_\lambda x_n\in S\cap
L^p(\mathcal M,\tau)$. Combining with (\ref{equa5.2.1}), we have
that $e_\lambda x\in [S\cap L^p(\mathcal{M},\tau)]_p $ for each
$\lambda \in \Lambda$. Thus $x\in [S\cap L^p(\mathcal{M},\tau)]_p $
for each $\lambda \in \Lambda$. Or,
$$ [S]_2\cap L^p(\mathcal{M},\tau) \subseteq[S\cap
          L^p(\mathcal{M},\tau)]_p.$$ This ends the proof of the
          proposition.

   \end{proof}

\begin{lemma}\label{lemma5.3.1}If $u$ is a partial isometry in
$\mathcal M$ such that $uu^*\in\mathcal D$, then
\begin{enumerate}
\item [(i)] $\displaystyle [(H^2 u)\cap L^p(\mathcal M,\tau)]_p=H^pu  \quad   \text{
for   }  0<p<1;$    \item [(ii)] $ \displaystyle H^2 u=
 [{H^pu\cap L^2(\mathcal M},\tau)]_2 \quad      \text{ for } 0<p<1 . $
\end{enumerate}
\end{lemma}
\begin{proof}
(i) Assume that $x\in  H^2 $ such that $xu\in (H^2 u)\cap
L^p(\mathcal M,\tau)$. Then $xuu^*\in L^p(\mathcal M,\tau)$, and
$x(uu^*)$ is also in $H^2$, as $uu^*\in\mathcal D$. So $xuu^*\in
H^2\cap L^p(\mathcal M,\tau)\subseteq H^p$ by Proposition 3.2 in
\cite{Be}. Note that $H^pu$ is a closed subspace in $L^p(\mathcal
M,\tau)$. We have
$$
[(H^2 u)\cap L^p(\mathcal M,\tau)]_p\subseteq H^pu.
$$
Similarly, we have
$$
[(H^p u)\cap L^2(\mathcal M,\tau)]_2 \subseteq H^2u.
$$
Combining with  Proposition \ref{proposition5.1}, we have $$
H^pu=[H^pu\cap L^2(\mathcal M,\tau)]_p\subseteq [(H^2 u)\cap
L^p(\mathcal M,\tau)]_p .$$ Hence $\displaystyle [(H^2 u)\cap
L^p(\mathcal M,\tau)]_p=H^pu, \text{ for   }  0<p<1.$

(ii) Let $x\in H^2$.  By Lemma
            \ref{lemma2.13},  we can find a net
$\{e_\lambda\}_{\lambda\in\Lambda}$ of projections in $ \mathcal{D}$
such that $\lim_\lambda \|e_\lambda x- x\|_2=0$ and $\tau
(e_\lambda)<\infty$ for each $\lambda \in \Lambda$. From $\tau
(e_\lambda)<\infty$, it is easy to verify that $e_\lambda x\in
L^p(\mathcal M,\tau)\cap H^2\subseteq H^p$   by Proposition 3.2 in
\cite{Be}. Thus $e_\lambda xu \in (H^p u)\cap L^2(\mathcal M,\tau)$
for each $\lambda\in\Lambda$, whence $xu\in [(H^p u)\cap
L^2(\mathcal M,\tau)]_2.$ Or,
$$ H^2u \subseteq [(H^p u)\cap L^2(\mathcal M,\tau)]_2  .
$$ Combining with what we proved in (i), we have $$ H^2u
= [(H^p u)\cap L^2(\mathcal M,\tau)]_2  .
$$
\end{proof}

        Now, we can prove a Beurling-Blecher-Labuschagne Theorem for the semifinite case when $0<p<1$.

    \begin{theorem} \label{theorem5.4} Let $0<p<1$.
Let $\mathcal M$ be a von Neumann algebra with a faithful, normal,
semifinite tracial weight $\tau$, and $H^\infty$  be  a semifinite
subdigonal subalgebra   of $\mathcal M$.  Let $\mathcal
D=H^\infty\cap (H^\infty)^*$. Assume that $\mathcal K  $ is a closed
subspace of $L^p(\mathcal M,\tau)$ such that $H^\infty \mathcal K
\subseteq \mathcal K$.

             Then there exist   a   closed subspace $Y$ of $L^p(\mathcal M,\tau)$ and a family $\{u_\lambda\}_{\lambda \in\Lambda}$ of partial isometries in
             $ \mathcal{M}$  such that:
            \begin{enumerate}
              \item  [(i)] $u_\lambda Y^*=0$ for all ${\lambda\in\Lambda}$.
\item  [(ii)] $u_\lambda u_\lambda^*\in \mathcal D$ and $u_\lambda u_\mu^*=0$
for all $\lambda, \mu\in \Lambda$ with $\lambda\ne \mu$.
\item  [(iii)] $Y= [H^\infty_0Y]_p$.
                \item [(iv)] $\mathcal K=Y \oplus^{row} (\oplus^{row}_{\lambda\in\Lambda } H^p u_\lambda)$
            \end{enumerate} Here $\oplus^{row}$ is the row sum of
            subspaces
            defined in Definition \ref{definition2.12}.
        \end{theorem}

    \begin{proof}
Let $\mathcal K_1= [\mathcal K \cap L^2(\mathcal M,\tau)]_2$. Then
$\mathcal K_1$ is a   closed subspace of $L^2(\mathcal M,\tau)$ such
that
  $H^\infty \mathcal K_1\subseteq \mathcal K_1$.  By Theorem \ref{theorem4.6}, there exist a  closed subspace $Y_1$ of $ L^2(\mathcal M,\tau)$ and a family $\{u_\lambda\}_{\lambda\in\Lambda}$ of partial
isometries in $\mathcal M$, satisfying
\begin{enumerate}
              \item  [(a)] $u_\lambda Y_1^*=0$ for all ${\lambda\in\Lambda}$.
\item  [(b)] $u_\lambda u_\lambda^*\in \mathcal D$ and $u_\lambda u_\mu=^*0$
for all $\lambda, \mu\in \Lambda$ with $\lambda\ne \mu$.
\item  [(c)] $Y_1=[H^2_0Y_1]_2$.
                \item [(d)] $\mathcal K_1=Y_1 \oplus^{row} (\oplus^{row}_{\lambda\in\Lambda } H^2 u_\lambda)$
\end{enumerate}
Let
$$
Y=[{Y_1\cap L^p(\mathcal{M},\tau)}]_p.
$$

(i) From (a), definition of $Y$ and Lemma \ref{Lemma2.5}, we can
conclude that
\begin{align}
\text{$u_\lambda Y^*=0$ for all
${\lambda\in\Lambda}$.\label{equa5.3}}
\end{align}

(ii)   follows directly from
 (b), i.e.
 \begin{equation}
  \text{ $u_\lambda u_\lambda^*\in \mathcal D$ and $u_\lambda u_\mu^*=0$
for all $\lambda, \mu\in \Lambda$ with $\lambda\ne \mu$.}
\label{equa5.4}
\end{equation}

(iii) We want to show that $Y= [H^2_0Y]_p$. First we will show that
\begin{align}
[(H_0^\infty Y_1)\cap L^p(\mathcal M,\tau)]_p\subseteq
[H_0^\infty(Y_1\cap L^p(\mathcal M,\tau))]_p\notag
\end{align}
In fact, let $x\in Y_1$ and $h\in H_0^\infty$ such that $hx\in
(H_0^\infty Y_1)\cap L^p(\mathcal M,\tau)$. We want to show that
$hx\in [H_0^\infty(Y_1\cap L^p(\mathcal M,\tau))]_p$. By Lemma
            \ref{lemma2.13},   we can find
a net $\{e_\lambda\}_{\lambda\in\Lambda}$ of projections in $
\mathcal{D}$ such that $e_\lambda\rightarrow I$ in weak$^*$-topology
and $\tau(e_\lambda)<\infty$ for each $\lambda\in\Lambda$.  By Lemma
\ref{Lemma2.6.2}, we have
\begin{align}\lim_\lambda \|e_\lambda hx-hx\|_p=0.
\label{equa5.4.2}\end{align} Thus, to show that $hx\in
[H_0^\infty(Y_1\cap L^p(\mathcal M,\tau))]_p$, it suffices to prove
that $e_\lambda hx\in [H_0^\infty(Y_1\cap L^p(\mathcal M,\tau))]_p$
for each $\lambda \in\Lambda$. Fix a $\lambda_0\in \Lambda$. Then,
for some positive number $q$ with $1/p=1/2+1/q$, we have
\begin{align}
\lim_\lambda \|e_{\lambda_0} h e_\lambda x -  e_{\lambda_0} h
x\|_p\le  \lim_\lambda \|e_{\lambda_0} h\|_q\| e_\lambda x - x\|_2=
0,\label{equa5.4.1}
\end{align} as $x\in Y_1$.  Moreover, we have
$e_{\lambda_0} h\in H_0^\infty$ and $e_\lambda x\in Y_1\cap
L^p(\mathcal M,\tau)$, as $\|e_\lambda x\|_p\le \|e_\lambda
\|_q\|x\|_2<\infty.$ Thus, $e_{\lambda_0} h e_\lambda x$ is in
$H_0^\infty(Y_1\cap L^p(\mathcal M,\tau))$ for each
$\lambda\in\Lambda$.  Whence, from (\ref{equa5.4.1}), $e_{\lambda_0} h x$ is in
$[H_0^\infty(Y_1\cap L^p(\mathcal M,\tau))]_p$ for each
$\lambda_0\in \Lambda$. Therefore, from (\ref{equa5.4.1}), $hx\in [H_0^\infty(Y_1\cap
L^p(\mathcal M,\tau))]_p$. Or
\begin{align}
[(H_0^\infty Y_1)\cap L^p(\mathcal
M,\tau)]_p\subseteq[H_0^\infty(Y_1\cap L^p(\mathcal
M,\tau))]_p\label{equa5.5}
\end{align}

Now,  we have
 \begin{align}
                Y &=[Y_1  \cap L^p(\mathcal{M},\tau)]_p\tag{by definition of $Y$}\\
                &=[[H_0^2 Y_1]_2\cap L^p(\mathcal{M},\tau)]_p\tag{by (c)}\\
                &=[(H_0^\infty Y_1)\cap L^p(\mathcal{M},\tau)]_p\tag{by Proposition \ref{proposition5.2}}\\
&\subseteq [H_0^\infty(Y_1\cap L^p(\mathcal M,\tau))]_p\tag{by (\ref{equa5.5})}\\
&\subseteq [H_0^\infty Y]_p\subseteq Y. \tag{by the definition of
$Y$}
            \end{align}
Thus,
 \begin{align}
                Y = [H_0^\infty Y]_p . \label{equa5.6}
            \end{align}
 (iv) We have only  to show that $$\mathcal K=Y \oplus^{row}
(\oplus^{row}_{\lambda\in\Lambda } H^p u_\lambda).$$  By   the
definition of $Y$, we have
\begin{align}
Y=[Y_1\cap L^p(\mathcal M,\tau)]_p .\label{equa5.13}
            \end{align} And from Lemma \ref{lemma5.3.1}, we have
\begin{align}
H^pu_\lambda =[H^2 u_\lambda \cap L^p(\mathcal M,\tau)]_p , \ \
\forall \ \lambda\in\Lambda.\label{equa5.14}
            \end{align} Now, we have
\begin{align}
\mathcal K&= [\mathcal K_1\cap L^p(\mathcal M,\tau)]_p \tag {by
Proposition \ref{proposition5.1}}\\
 &=[ [span\{Y_1,
H^2 u_\lambda : \lambda \in\Lambda \}]_2\cap L^p(\mathcal
M,\tau) ]_p \tag{by the definition of row sum of subspaces} \\
&=[
 {span\{Y_1, H^2 u_\lambda : \lambda \in\Lambda \}} \cap
L^p(\mathcal M,\tau) ]_p\tag{by  Proposition \ref{proposition5.2}}\\
&=[
 {span\{Y_1\cap
L^p(\mathcal M,\tau), H^2 u_\lambda \cap L^p(\mathcal M,\tau):
\lambda \in\Lambda \}} ]_p\tag{by   (a) and (b)}\\
&  =[
 {span\{Y  , H^p u_\lambda  :
\lambda \in\Lambda \}} ]_p \tag  {by (\ref{equa5.13}) and
(\ref{equa5.14}) } \\& =Y \oplus^{row}
(\oplus^{row}_{\lambda\in\Lambda } H^p u_\lambda), \label{equa5.9}
\end{align}  where the last equation follows from the definition of the row sum of
  subspaces.

As a summary, from (\ref{equa5.3}), (\ref{equa5.4}),
(\ref{equa5.6}), and (\ref{equa5.9}), $Y$ and
$\{u_\lambda\}_{\lambda \in\Lambda}$ have desired properties. This
ends the proof of the theorem.
    \end{proof}

\begin{corollary}\label{corollary5.5} Let $\mathcal M$ be a von Neumann algebra with a faithful, normal,
semifinite tracial weight $\tau$.
\begin{enumerate}

\item [(i)] Let $0<p<\infty$. If $\mathcal K$ is a  closed subspace of $L^p(\mathcal M, \tau)$
such that $\mathcal M \mathcal K\subseteq \mathcal K$, then there
exists a projection $q\in \mathcal M$ such that $\mathcal K=
L^p(\mathcal M, \tau) q.$
\item [(ii)] If $\mathcal K$ is a weak$^*$-closed subspace of $\mathcal M$
such that $\mathcal M \mathcal K\subseteq \mathcal K$, then there
exists a projection $q\in \mathcal M$ such that $\mathcal K=
\mathcal M q.$

\end{enumerate}
\end{corollary}

\begin{proof} (i)
Note that $\mathcal{M}$ itself is a semifinite subdiagonal
subalgebra of $\mathcal{M}$. Let $H^{\infty}=\mathcal{M}$. Then
$\mathcal{D}=\mathcal{M}$ and $\Phi$ is the identity map from
$\mathcal{M}$ to  $\mathcal{M}$. Hence $H_{0}^{\infty}=\{0\}$ and
$H^{ p}=  L^{p}(\mathcal{M},\tau)$.

Assume that $\mathcal{K}$ is a closed subspace of $L^p(\mathcal{M}%
,\tau) $ such that $\mathcal{K}\mathcal{M}\subseteq \mathcal{K}$.
From Theorem \ref{theorem4.6} and Theorem \ref{theorem5.4},
\[
\mathcal{K}= Y\oplus^{row}(\oplus^{row}_{\lambda\in \Lambda%
}H^{p}u_{\lambda}),
\]
where $Y$ and the $\{u_{\lambda}\}_{\lambda\in\Lambda}$ satisfy the
conditions in Theorem \ref{theorem4.6} and Theorem \ref{theorem5.4}.

From the fact that $H_{0}^{\infty}=\{0\}$, we know that $Y=\{0\}$.
Since $\mathcal{D}=\mathcal{M}$, we know that
\[
 H^{p}u_{\lambda}=  L^{p}(\mathcal{M},\tau)u_{\lambda} = L^{p}(\mathcal{M},\tau) u_{\lambda}u^{*}%
_{\lambda} u_{\lambda}\subseteq  L^{p }(\mathcal{M},\tau) u^{*}%
_{\lambda} u_{\lambda}\subseteq L^{p}(\mathcal{M},\tau)u_{\lambda}=
H^{p} u_{\lambda}.
\]
So $H^{p}u_{\lambda}=  L^{p }(\mathcal{M},\tau) u^{*}%
_{\lambda} u_{\lambda}$ and
\[
\mathcal{K}= Y\oplus^{row}(\oplus^{row}_{\lambda\in \Lambda%
}H^{p}u_{\lambda}) = (\oplus^{row}_{\lambda\in \Lambda%
} L^{p }(\mathcal{M},\tau) u^{*}%
_{\lambda} u_{\lambda})=  L^{p}(\mathcal{M},\tau) \left
(\sum_{\lambda\in\Lambda } u^{*}%
_{\lambda} u_{\lambda}  \right )=  L^{p}(\mathcal{M},\tau)q,
\]
where $q= \sum_{\lambda\in\Lambda } u^{*}%
_{\lambda} u_{\lambda} $ is a projection in $\mathcal{M}$. This ends
the proof of (i).

 (ii) The proof is similar to (i).
\end{proof}

\section{Invariant Subspaces for Analytic Crossed Products }

\subsection{Crossed product  of a von Neumann algebra $\mathcal M$  by an action $\alpha$}\label{subsect6.1}

Let $\mathcal M$ be a von Neumann algebra with a semifinite, faithful, normal tracial state $\tau$.
Let $\alpha$ be a  trace-preserving $*$-automorphism  of $\mathcal M$ (so  $\tau(\alpha(x))=\tau(x), \  \forall x\in \mathcal M^+$).

We let $l^2(\mathbb Z)$ be the Hilbert space consisting of complex-valued functions $f$ on $\mathbb Z$ such that $\sum_{m\in \mathbb Z} |f(m)|^2<\infty$. We denote by $\{e_n\}_{n\in\mathbb Z}$ the orthonormal basis of $l^2(\mathbb Z)$ determined by $e_n(m)=\delta(n,m)$. We also denote by $\lambda: \mathbb Z\rightarrow B(l^2(\mathbb Z))$ the left regular representation of $\mathbb Z$ on $l^2(\mathbb Z)$, i.e. each $\lambda(n)$ is determined by  $\lambda(n)(e_m)=e_{m+n}$.

Let $\mathcal H=L^2(\mathcal M,\tau)\otimes l^2(\mathbb Z)$. Then $\mathcal H$ can also be written as $\oplus_{m\in \mathbb Z} L^2(\mathcal M,\tau)\otimes e_m$. Consider representations $\Psi$ and $\Lambda$ of $\mathcal M$ and $\mathbb Z$, respectively, on $\mathcal H$, defined by
$$
\begin{aligned}
  \Psi(x) (\xi\otimes e_m) &= (\alpha^{-m}(x)\xi)\otimes e_m, \qquad  \forall \ x\in\mathcal M, \  \forall \ \xi\in L^2(\mathcal M,\tau), \ \forall \  m\in \mathbb Z\\
  \Lambda (n)(\xi\otimes e_m)&=\xi\otimes (\lambda(n)e_m), \qquad  \qquad \ \forall  \ n,m \in \mathbb Z
\end{aligned}
$$
It can be verified that
$$
\Lambda(n) \Psi(x)\Lambda(-n)=\Psi(\alpha^n(x)), \qquad \forall \  x\in\mathcal M, \ \forall  \ n  \in \mathbb Z.
$$
Then the crossed product of $\mathcal M$ by an action $\alpha$, denoted by $\mathcal M\rtimes_\alpha \mathbb Z$, is the von Neumann algebra generated by $\Psi(\mathcal M)$ and $\Lambda(\mathbb Z)$ in $B(\mathcal H)$. If no confusion arises, we will identify $\mathcal M$ with its image $\Psi(\mathcal M)$ in $\mathcal M\rtimes_\alpha \mathbb Z$.

It is well known (for example, see Chapter 13 in \cite{KR}) that there exists a faithful, normal conditional expectation $\Phi$ from $\mathcal M\rtimes_\alpha \mathbb Z$ onto $\mathcal M$ such that
$$
\Phi\left (\sum_{n=-N}^N  \Lambda (n)\Psi(x_n)\right )=x_0, \qquad \text{where }  \ x_n\in\mathcal M \text { for all } -N\le n\le N.
$$
Moreover, there exists a semifinite, faithful, normal, extended tracial weight, still denoted by $\tau$, on  $\mathcal M\rtimes_\alpha \mathbb Z$ satisfying
$$
\tau (y)=\tau(\Phi(y)), \qquad  \text {for every positive element } y \text{ in }  \mathcal M\rtimes_\alpha \mathbb Z.
$$
\begin{example}\label{exam6.1}
Let $\mathcal M=l^\infty(\mathbb Z)$ be an   abelian von Neumann algebra with a semifinite, faithful, normal tracial weight $\tau$, determined by
$$
\tau(f)=\sum_{m\in\mathbb Z} f(m), \qquad \text { for every positive element } \ f\in l^\infty(\mathbb Z).
$$ Let $\alpha$ be an action on $l^\infty(\mathbb Z)$, defined by
$$
\alpha(f)(m)= f(m-1), \qquad \text { for every element } \ f\in l^\infty(\mathbb Z).
$$ It is not hard to verify (for example see Proposition 8.6.4 in \cite{KR}) that $l^\infty(\mathbb Z)\rtimes_\alpha \mathbb Z$  is a type I$_\infty$ factor. Thus $l^\infty(\mathbb Z)\rtimes_\alpha \mathbb Z \simeq B(\mathcal H)$ for some separable Hilbert space $\mathcal H$.
\end{example}

\subsection{Invariant subspace for crossed products}
From the construction of crossed product, we have the following result immediately (also see Section 3 in \cite{Arv}).
\begin{lemma}\label{lemma6.1}
Let $
\mathcal M\rtimes_\alpha \mathbb Z_+
$ be a weak $*$-closed non-self-adjoint subalgebra generated by
$$
\{\Lambda (n)\Psi(x) : x\in \mathcal M, \ n\ge 0\}
$$ in  $\mathcal M\rtimes_\alpha \mathbb Z$.  Then the following statements are true:
\begin{enumerate}
  \item [(i)] $
\mathcal M\rtimes_\alpha \mathbb Z_+
$  is a semifinite subdiagonal subalgebra with respect to   $(\mathcal M\rtimes_\alpha \mathbb Z,\Phi).$ (Such $
\mathcal M\rtimes_\alpha \mathbb Z_+
$  is called an analytic crossed product and will be denoted     by $H^\infty$.)
  \item [(ii)] $H^\infty_0=ker(\Phi)\cap H^\infty$ is a weak $*$-closed nonself-adjoint subalgebra generated by
$$
\{\Lambda (n)\Psi(x)  : x\in \mathcal M, \ n> 0\}
$$ in  $\mathcal M\rtimes_\alpha \mathbb Z$ satisfying
$$
H^\infty_0=\Lambda (1)H^\infty.
$$
\item [(iii)] $H^\infty\cap (H^\infty)^*=\mathcal M$.
\end{enumerate}
\end{lemma}

  Following the notation in Section \ref{subsect6.1}, our next result   characterizes invariant subspaces in a crossed product of a semifinite von Neumann algebra $\mathcal M$ by a tracing-preserving action $\alpha$.

\begin{theorem}\label{theorem6.2}
 Let $\mathcal M$ be a von Neumann algebra with a semifinite, faithful, normal tracial weight $\tau$, and  $\alpha$ be a  trace-preserving $*$-automorphism  of $\mathcal M$. Denote by  $ \mathcal M\rtimes_\alpha \mathbb Z$ the crossed product of $\mathcal M$ by an action $\alpha$, and still denote by $\tau$ a  semifinite, faithful, normal, extended tracial weight on  $\mathcal M\rtimes_\alpha \mathbb Z$.

Let     $H^\infty $ be a weak $*$-closed non-self-adjoint subalgebra generated by
$
\{\Lambda (n)\Psi(x)  : x\in \mathcal M, \ n\ge 0\}
$  in  $\mathcal M\rtimes_\alpha \mathbb Z$, be a semifinite  subdiagonal subalgebra of  $\mathcal M\rtimes_\alpha \mathbb Z$.
 Then the following statements are true.

\begin{enumerate}
\item [(i)] Let $0<p<\infty$.  Assume that $\mathcal K  $ is a closed
subspace of $L^p(\mathcal M\rtimes_\alpha \mathbb Z,\tau)$ such that $H^\infty \mathcal K
\subseteq \mathcal K$.
             Then there exist a projection $q$ in $\mathcal M$  and    a    family $\{u_\lambda\}_{\lambda \in\Lambda}$ of partial isometries in
             $ \mathcal M\rtimes_\alpha \mathbb Z$  satisfying
            \begin{enumerate}
              \item [(a)]  $u_\lambda q=0$  for all $\lambda \in \Lambda$;
              \item  [(b)]  $u_\lambda u_\lambda^*\in \mathcal M$ and $u_\lambda u_\mu^*=0$
for all $\lambda, \mu\in \Lambda$ with $\lambda\ne \mu$;
                \item [(c)] $\mathcal K=(L^p( \mathcal M\rtimes_\alpha \mathbb Z,\tau)q) \oplus^{row}(  \oplus^{row}_{\lambda\in\Lambda } H^p u_\lambda).$
            \end{enumerate}
\item [(ii)]   Assume that $\mathcal K  $ is a weak $*$-closed
subspace of $ \mathcal M\rtimes_\alpha \mathbb Z $ such that $H^\infty \mathcal K
\subseteq \mathcal K$.
             Then there  exist a projection $q$ in $\mathcal M$  and   a    family $\{u_\lambda\}_{\lambda \in\Lambda}$ of partial isometries in
             $ \mathcal M\rtimes_\alpha \mathbb Z$  satisfying
            \begin{enumerate}
             \item [(a)]  $u_\lambda q=0$  for all $\lambda \in \Lambda$;
              \item  [(b)]  $u_\lambda u_\lambda^*\in \mathcal M$ and $u_\lambda u_\mu^*=0$
for all $\lambda, \mu\in \Lambda$ with $\lambda\ne \mu$;
                \item [(c)] $\mathcal K= ( ( \mathcal M\rtimes_\alpha \mathbb Z)q )\oplus^{row}( \oplus^{row}_{\lambda\in\Lambda } H^\infty u_\lambda).$
            \end{enumerate}
\end{enumerate}

\end{theorem}
\begin{proof}
(i)
From Theorem \ref{theorem4.6} and Theorem \ref{theorem5.4},
\[
\mathcal{K}= Y\oplus^{row}(\oplus^{row}_{\lambda\in \Lambda%
}H^{p}u_{\lambda}),
\]
where $Y$ is a closed subspace of $\mathcal M\rtimes_\alpha \mathbb Z$ and   $\{u_{\lambda}\}_{\lambda\in\Lambda}$  is a family of partial isometries in $\mathcal M\rtimes_\alpha \mathbb Z $ satisfying   \begin{enumerate}
              \item [(a$_1$)]   $u_\lambda Y^*=0$ for all ${\lambda\in\Lambda}$;
\item  [(b$_1$)]      $u_\lambda u_\lambda^*\in \mathcal M$ and $u_\lambda u_\mu^*=0$
for all $\lambda, \mu\in \Lambda$ with $\lambda\ne \mu$;
\item   [(c$_1$)]     $Y= [H^\infty_0Y]_p$.
            \end{enumerate} From   (c$_1$)    and  Lemma \ref{lemma6.1}, we have
$$
Y=[H^\infty_0 Y]_p= [\Lambda (1)H^\infty Y]_p\subseteq \Lambda(1)Y.
$$ So, $Y$ is a left $\mathcal M\rtimes_\alpha \mathbb Z$-invariant subspace of  $L^p(\mathcal M\rtimes_\alpha \mathbb Z,\tau)$. From Corollary \ref{corollary5.5}, there exists a projection $q$ in $\mathcal M $ such that $Y=  L^p( \mathcal M\rtimes_\alpha \mathbb Z,\tau)q.$ Therefore, we have
            \begin{enumerate}
              \item [(a)]  $u_\lambda q=0$  for all $\lambda \in \Lambda$;
              \item  [(b)]  $u_\lambda u_\lambda^*\in \mathcal M$ and $u_\lambda u_\mu^*=0$
for all $\lambda, \mu\in \Lambda$ with $\lambda\ne \mu$;
                \item [(c)] $\mathcal K=(L^p( \mathcal M\rtimes_\alpha \mathbb Z,\tau)q) \oplus^{row}(  \oplus^{row}_{\lambda\in\Lambda } H^p u_\lambda).$
            \end{enumerate}
This ends the proof of (i).

(ii) The proof is similar to (i).
\end{proof}

\subsection{Invariant subspaces for Schatten $p$-classes}
Let $\mathcal H$ be an infinite dimensional  separable Hilbert space with an orthonormal base  $\{e_m\}_{m\in\mathbb Z}$.
 Let $\tau=Tr$ be the usual trace on $B(\mathcal H)$, i.e.
  $$
  \tau(x)=\sum_{i\in\mathbb Z} \langle xe_m, e_m\rangle , \qquad \text{ for all positive } x   \text { in }    B(\mathcal H).
  $$Then $B(\mathcal H)$ is a von Neumann algebra with a semifinite, faithful, normal tracial weight $\tau$.
 For each $0<p<\infty$,  the Schatten $p$-class $S^p(\mathcal H)$ is  the associated non-commuative $L^p$-space  $L^p(B(\mathcal H),\tau)$.

Let
$$
\mathcal A =\{ x\in B(\mathcal H) : \langle xe_m, e_n\rangle =0, \ \ \forall n<m\}
$$ be the lower triangular subalgebra of $B(\mathcal H)$. 
From Example \ref{exam6.1}, $B(\mathcal H)$ can also be realized as a crossed product $l^\infty(\mathbb Z) \rtimes_\alpha \mathbb Z$ of $l^\infty(\mathbb Z)$ by an action $\alpha$, where the action $\alpha$ is determined by
$$
\alpha(f)(m)=f(m-1), \qquad \forall \ f\in l^\infty(\mathbb Z).
$$ Moreover, it can be verified quickly that $\mathcal A$, as a subalgebra of $B(\mathcal H)$, is $ l^\infty(\mathbb Z) \rtimes_\alpha \mathbb Z_+ $
(see  Lemma \ref{lemma6.1}) is  a semifinite subdiagonal subalgebra of $l^\infty(\mathbb Z) \rtimes_\alpha \mathbb Z $ (see Example 2.6 in \cite{MMS}). Thus from Theorem \ref{theorem6.2}, we have the following statements.

\begin{corollary} \label{cor6.3}
Let $\mathcal H$ be a separable Hilbert space with  an orthonormal base $\{e_m\}_{m\in\mathbb Z}$.
 Let $H^\infty$ be the lower triangular subalgebra of $B(\mathcal H)$, i.e.
$$
H^\infty =\{ x\in B(\mathcal H) : \langle xe_m, e_n\rangle =0, \ \ \forall n<m\}.
$$
 Let $\mathcal D=H^\infty\cap (H^\infty)^*$ be the diagonal subalgebra of $B(\mathcal H)$.
\begin{enumerate}\item [(i)]  For each $0<p< \infty$, let $S^p(\mathcal H)$  be the Schatten $p$-class.  Assume that $\mathcal K  $ is a closed
subspace of $S^p(\mathcal H)$ such that $H^\infty \mathcal K
\subseteq \mathcal K$.
             Then there exist a   projection $q$ in $\mathcal D$  and    a    family $\{u_\lambda\}_{\lambda \in\Lambda}$ of partial isometries in
             $ B(\mathcal H)$  satisfying
            \begin{enumerate}
              \item [(a)]  $u_\lambda q=0$  for all $\lambda \in \Lambda$;
              \item  [(b)]  $u_\lambda u_\lambda^*\in  \mathcal D  $   and $u_\lambda u_\mu^*=0$
for all $\lambda, \mu\in \Lambda$ with $\lambda\ne \mu$;
                \item [(c)] $\mathcal K=(S^p(\mathcal H)q) \oplus^{row}(  \oplus^{row}_{\lambda\in\Lambda } H^p u_\lambda).$
            \end{enumerate}
\item [(ii)]   Assume that $\mathcal K  $ is a weak $*$-closed
subspace of $ B(\mathcal H) $ such that $H^\infty \mathcal K
\subseteq \mathcal K$.
             Then there  exist a projection $q$ in $\mathcal D$  and   a    family $\{u_\lambda\}_{\lambda \in\Lambda}$ of partial isometries in
             $ B(\mathcal H)$  satisfying
            \begin{enumerate}
             \item [(a)]  $u_\lambda q=0$  for all $\lambda \in \Lambda$;
              \item  [(b)]  $u_\lambda u_\lambda^*\in \mathcal D$ and $u_\lambda u_\mu^*=0$
for all $\lambda, \mu\in \Lambda$ with $\lambda\ne \mu$;
                \item [(c)] $\mathcal K= ( B(\mathcal H)q )\oplus^{row}( \oplus^{row}_{\lambda\in\Lambda } H^\infty u_\lambda).$
            \end{enumerate}
\end{enumerate}

\end{corollary}

\begin{remark}Let   $0<p< \infty$.
If $q$ is a projection in $\mathcal D$ such that $S^p(\mathcal H)q\subseteq H^p$, then $q=0$.
\end{remark}

The the next result follows directly from Corollary \ref{cor6.3}.

\begin{corollary} \label{cor6.4}
Let $\mathcal H$ be a separable Hilbert space with  an orthonormal base  $\{e_m\}_{m\in\mathbb Z}$.
 Let $H^\infty$ be the lower triangular subalgebra of $B(\mathcal H)$, i.e.
$$
H^\infty =\{ x\in B(\mathcal H) : \langle xe_m, e_n\rangle =0, \ \ \forall n<m\}.
$$ Let $\mathcal D=H^\infty\cap (H^\infty)^*$ be the diagonal subalgebra of $B(\mathcal H)$.
\begin{enumerate}\item [(i)]  For each $0<p< \infty$, if $\mathcal K  $ is a closed
subspace of $H^p$ such that $H^\infty \mathcal K
\subseteq \mathcal K$,
             then there exists   a    family $\{u_\lambda\}_{\lambda \in\Lambda}$ of partial isometries in
             $ H^\infty$  satisfying
            \begin{enumerate}
              \item  [(a)]  $u_\lambda u_\lambda^*\in  \mathcal D  $   and $u_\lambda u_\mu^*=0$
for all $\lambda, \mu\in \Lambda$ with $\lambda\ne \mu$;
                \item [(b)] $\mathcal K=  \oplus^{row}_{\lambda\in\Lambda } H^p u_\lambda .$
            \end{enumerate}
\item [(ii)]   Assume that $\mathcal K  $ is a weak $*$-closed
subspace of $ H^\infty $ such that $H^\infty \mathcal K
\subseteq \mathcal K$.
             Then there  exists     a    family $\{u_\lambda\}_{\lambda \in\Lambda}$ of partial isometries in
             $ H^\infty$  satisfying
            \begin{enumerate}
              \item  [(a)]  $u_\lambda u_\lambda^*\in \mathcal D$ and $u_\lambda u_\mu^*=0$
for all $\lambda, \mu\in \Lambda$ with $\lambda\ne \mu$;
                \item [(b)] $\mathcal K=  \oplus^{row}_{\lambda\in\Lambda } H^\infty u_\lambda .$
            \end{enumerate}
\end{enumerate}

\end{corollary}

\begin{remark}
Similar results hold  when $H^\infty$ is the upper triangular subalgebra of $B(\mathcal H)$.
\end{remark}


\begin{thebibliography}{99}                                                                                               %
\bibitem {Arv}W. B. Arveson, \emph{Analyticity in operator algebras}, Amer. J.
Math. 89 (1967) 578--642.

\bibitem {Be} T. Bekjan, \emph{ Noncommutative Hardy space associated with semi-finite subdiagonal algebras}, J. Math. Anal. Appl. 429 (2015), no. 2,
1347–1369.

\bibitem {BX}T. N. Bekjan and Q. Xu, \emph{Riesz and Szeg\"{o} type
factorizations for noncommutative Hardy spaces}, J. Operator Theory,
62 (2009) 215-231.

 \bibitem {Be1}T. N. Bekjan, \emph{Noncommutative symmetric Hardy spaces}, Integr.
 Equ. Oper. Theory 81 (2015) 191-212.

 \bibitem {Be}T. N. Bekjan, \emph{ Noncommutative Hardy space associated with
 semi-finite subdiagonal algebras}, J. Math. Anal. Appl. 429 (2015),
 no. 2, 1347-1369.



\bibitem {B}A. Beurling, \emph{On two problems concerning linear
transformations in Hilbert space}, Acta Math. 81 (1949) 239-255.


\bibitem {BL2}D. Blecher and L. E. Labuschagne, \emph{A Beurling theorem for
noncommutative $L^{p}$}, J. Operator Theory, 59 (2008) 29-51.



 \bibitem {Bo}S. Bochner, \emph{Generalized conjugate and analytic functions
 without expansions}, Proc. Nat. Acad. Sci. U.S.A. 45 (1959) 855-857.



\bibitem {P2}P. Dodds and T. Dodds, \emph{Some properties of symmetric
operator spaces}, Proc. Centre Math. Appl. Austral. Nat. Univ., 29,
Austral. Nat. Univ., Canberra, 1992.

\bibitem {P}P. Dodds, T. Dodds and B. Pagter, \emph{Noncommutative Banach
function spaces}, Mathematische Zeitschrift, 201 (1989) 583-597.



\bibitem {Do}P. Dodds, T. Dodds and B. Pagter, \emph{Noncommutative K\"{o}the
duality}, Trans. Amer. Math. Soc. 339 (1993) 717-750.

\bibitem {Exel}R. Exel, \emph{Maximal subdiagonal algebras}, Amer. J. Math.
110 (1988) 775-782.

\bibitem {Fa}T. Fack and H. Kosaki, \emph{Generalized $s$-numbers of $\tau$-measurable
operators}, Pacific J. Math. 2 (1986) 269-300.



\bibitem {Fang}J. Fang, D. Hadwin, E. Nordgren and J. Shen, \emph{Tracial
gauge norms on finite von Neumann algebras satisfying the weak
Dixmier property}, J. Funct. Anal. 255 (2008) 142-183.



\bibitem {Halmos}P. Halmos, \emph{Shifts on Hilbert spaces}, J. Reine Angew.
Math. 208 (1961) 102-112.

\bibitem {He}H. Helson, \emph{Lectures on Invariant Subspaces}, Academic
Press, New York-London, 1964.

\bibitem {HL}H. Helson and D. Lowdenslager, \emph{Prediction theory and
Fourier series in several variables}, Acta Math. 99 (1958) 165-202.

\bibitem {Ho}K. Hoffman, \emph{Analytic functions and logmodular Banach
algebras}, Acta Math. 108 (1962) 271-317.


\bibitem {KR} R. Kadison and J. Ringrose, \emph{Fundamentals of the thoery of operator algebras, volume II, advanced theory}, Academic Press, Inc, (1986).

\bibitem {Kunze}R. A. Kunze, \emph{$L^{p}$-Fourier transforms on locally
compact unimodular groups}, Trans. Amer. Math. Soc. 89 (1958)
519-540.

\bibitem{Ji} G. Ji, \emph{Maximality of semi-finite subdiagonal algebras}, J. Shaanxi Normal Univ. Sci. Ed. 28 (2000) 15-17.

\bibitem {JS}M. Junge and D. Sherman, \emph{Noncommutative $L^{p}$-modules},
J. Operator Theory 53 (2005) 3-34.

\bibitem {MG}M. Marsalli and G. West, \emph{Noncommutative $H^{p}$ spaces },
J. Operator Theory  40 (1998) 339-355.



\bibitem {McCarthy}C. A. McCarthy, \emph{$C_{p}$}, Israel J. Math. 5 (1967) 249-271.

\bibitem {MMS} M. McAsey, P. Muhly and K. Saito, \emph{Nonselfadjoint Crossed Products (Invariant Subspaces and Maximality)},  Transactions of the American Mathematical Society, Vol. 248, No. 2 (Mar., 1979), pp. 381
-409.


\bibitem  {NW} T. Nakazi and Y. Watatani, {\em Invariant subspace theorems for subdiagonal algebras}, J. Operator Theory 37(1997), 379-395.


\bibitem {Nelson}E. Nelson, \emph{Notes on noncommutative integration}, J.
Funct. Anal. 15 (1974) 103-116.




\bibitem {vNeumann}J. von Neumann, \emph{Some matrix-inequalities and
metrization of matric-space}, Tomsk Univ. Rev. 1 (1937) 286-300.


\bibitem {PX}G. Pisier and Q. Xu, \emph{Noncommutative $L^{p}$-spaces},
Handbook of the geometry of Banach spaces, North-Holland, Amsterdam,
2 (2003) 1459-1517.
\bibitem {Sakai} Sakai, \emph{C$^*$-algebras and W$^*$-algebras},
Springer, 1971.

\bibitem {Sai}K. S. Saito, \emph{A note on invariant subspaces for finite
maximal subdiagonal algebras}, Proc. Amer. Math. Soc. 77 (1979)
348-352.


\bibitem {Sai2}  K. S. Saito, \emph{A simple approach to the invariant subspace structure of analytic crossed products,}
 J. Operator Theory 27 (1992), no. 1, 169-177.




\bibitem {Se}I. Segal, \emph{A noncommutative extension of abstract
integration}, Ann. Math. 57 (1952) 401-457.

\bibitem {Simon}B. Simon, \emph{Trace ideals and their applications}, London
Mathematical Society Lecture Note Series, vol. 35, Cambridge
University Press, Cambridge-New York, 1979.

\bibitem {Sr}T. P. Srinivasan, \emph{Simply invariant subspaces}, Bull. Amer.
Math. Soc. 69 (1963) 706-709.

\bibitem {SW}T. Srinivasan and J. K. Wang, \emph{Weak*-Dirichlet algebras},
Proceedings of the International Symposium on Function Algebras,
Tulane University, 1965 (Chicago), Scott-Foresman, 1966, 216--249.

\bibitem {Ta}M. Takesaki, \emph{Theory of Operator Algebras I}, Springer, 1979.

\bibitem{Xu} Q. Xu, \emph{On the maximality of subdiagonal algebras}, J. Operator Theory 54 (2005), no. 1, 137-146.

\bibitem{Xu2}  Q. Xu \emph{Operator spaces and noncommutative
L$^p$, The part on noncommutative L$^p$-spaces} Lectures in the
Summer School on Banach spaces and Operator spaces, Nankai
University China July 16 - July 20, 2007.

\bibitem {Y}F. Yeadon, \emph{Noncommutative $L^{p}$-spaces}, Math. Proc.
Cambridge Philos. Soc. 77 (1975) 91-102.
\end{thebibliography}
\end{document}